\documentclass[11pt]{amsart}
\usepackage[T1]{fontenc}
\usepackage[french,english]{babel}
\usepackage{lmodern}
\usepackage{amsmath, amsthm, amssymb, amsfonts}
\usepackage[normalem]{ulem}
\usepackage{mathrsfs}
\usepackage{verbatim} 
\usepackage{longtable}
\usepackage{mathtools}

\usepackage{tikz}
\usetikzlibrary{decorations.pathmorphing}
\tikzset{snake it/.style={decorate, decoration=snake}}
\usepackage{caption}
\usepackage{tikz-cd}
\usetikzlibrary{arrows}
\usetikzlibrary{decorations.pathmorphing}
\usepackage{hyperref}

\theoremstyle{plain}
\newtheorem{thm}{Theorem}[section]
\newtheorem{cor}[thm]{Corollary}
\newtheorem{lem}[thm]{Lemma}
\newtheorem{prop}[thm]{Proposition}
\newtheorem{conj}[thm]{Conjecture}
\newtheorem{question}[thm]{Question}

\theoremstyle{definition}
\newtheorem{defn}[thm]{Definition}

\theoremstyle{remark}
\newtheorem{rmk}[thm]{Remark}

\newcommand{\BA}{{\mathbb{A}}}

\newcommand{\BC}{{\mathbb{C}}}

\newcommand{\BG}{{\mathbb{G}}}

\newcommand{\BQ}{{\mathbb{Q}}}

\newcommand{\BZ}{{\mathbb{Z}}}

\newcommand{\CA}{{\mathcal A}}

\newcommand{\CE}{{\mathcal E}}
\newcommand{\CF}{{\mathcal F}}

\newcommand{\CH}{{\mathcal H}}

\newcommand{\CJ}{{\mathcal J}}

\newcommand{\CL}{{\mathcal L}}

\newcommand{\CN}{{\mathcal N}}
\newcommand{\CO}{{\mathcal O}}
\newcommand{\CP}{{\mathcal P}}

\newcommand{\CU}{{\mathcal U}}

\newcommand{\Fp}{{\mathfrak{p}}}
\newcommand{\Fq}{{\mathfrak{q}}}
\newcommand{\Fr}{{\mathfrak{r}}}

\newcommand{\Ft}{{\mathfrak{t}}}

\newcommand{\FC}{{\mathfrak{C}}}
\newcommand{\FF}{{\mathfrak{F}}}
\newcommand{\FK}{{\mathfrak{K}}}
\newcommand{\FZ}{{\mathfrak{Z}}}

\newcommand{\sslash}{\mathbin{/\mkern-6mu/}}

\newcommand{\td}{{\mathrm{td}}}

\DeclareFontFamily{OT1}{rsfs}{}
\DeclareFontShape{OT1}{rsfs}{n}{it}{<-> rsfs10}{}
\DeclareMathAlphabet{\curly}{OT1}{rsfs}{n}{it}

\newcommand\Hom{\operatorname{Hom}}

\newcommand{\Chow}{\mathrm{CH}}
\newcommand{\Corr}{\mathrm{Corr}}

\makeatletter
\let\@wraptoccontribs\wraptoccontribs
\makeatother

\newcommand{\newabstract}[1]{%
  \par\bigskip
  \csname otherlanguage*\endcsname{#1}%
  \csname captions#1\endcsname
  \item[\hskip\labelsep\scshape\abstractname.]
}

\begin{document}
\title[Algebraic cycles and Hitchin systems]{Algebraic cycles and Hitchin systems \\
Cycles alg\'ebriques et syst\`emes de Hitchin}
\date{\today}

\author[D. Maulik]{Davesh Maulik}
\address{Massachusetts Institute of Technology}
\email{maulik@mit.edu}

\author[J. Shen]{Junliang Shen}
\address{Yale University}
\email{junliang.shen@yale.edu}

\author[Q. Yin]{Qizheng Yin}
\address{Peking University}
\email{qizheng@math.pku.edu.cn}

\begin{abstract}
The purpose of this paper is to study motivic aspects of the Hitchin system for~$\mathrm{GL}_n$. Our results include the following. (a) We prove the motivic decomposition conjecture of Corti--Hanamura for the Hitchin system; in particular, the decomposition theorem associated with the Hitchin system is induced by algebraic cycles. This yields an unconditional construction of the motivic perverse filtration for the Hitchin system, which lifts the cohomological/sheaf-theoretic perverse filtration. (b) We prove that the inverse of the relative hard Lefschetz symmetry is induced by a relative algebraic correspondence, confirming the relative Lefschetz standard conjecture for the Hitchin system. (c) We show a strong perversity bound for the normalized Chern classes of a universal bundle with respect to the motivic perverse filtration; this specializes to the sheaf-theoretic result obtained earlier by Maulik--Shen. (d) We prove a $\chi$-independence result for the relative Chow motives associated with Hitchin systems.

Our methods combine Fourier transforms for compactified Jacobian fibrations associated with integral locally planar curves, nearby and vanishing cycle techniques, and a Springer-theoretic interpretation of parabolic Hitchin moduli spaces.

\newabstract{french}
L'objectif de cet article est d'\'etudier des aspects motiviques du syst\`eme de Hitchin pour $\mathrm{GL}_n$. Nos r\'esultats incluent les suivants. (a) Nous prouvons la conjecture de d\'ecomposition motivique de Corti--Hanamura pour le syst\`eme de Hitchin ; en particulier, le th\'eor\`eme de d\'ecomposition associ\'e au syst\`eme de Hitchin est induit par des cycles alg\'ebriques. Cela fournit une construction inconditionnelle de la filtration perverse motivique pour le syst\`eme de Hitchin, qui rel\`eve la filtration perverse cohomologique/faisceautique. (b) Nous prouvons que l'inverse de la sym\'etrie de Lefschetz difficile relative est induit par une correspondance alg\'ebrique relative, ce qui confirme la conjecture standard de Lefschetz relative pour le syst\`eme de Hitchin. (c) Nous montrons une borne de perversit\'e forte pour les classes de Chern normalis\'ees d'un fibr\'e universel par rapport \`a la filtration perverse motivique ; cela se restreint au r\'esultat faisceautique obtenu pr\'ec\'edemment par Maulik--Shen. (d) Nous prouvons un r\'esultat de $\chi$-ind\'ependance pour les motifs de Chow relatifs associ\'es aux syst\`emes de Hitchin.

Nos m\'ethodes combinent les transform\'ees de Fourier pour les fibrations jacobiennes compactifi\'ees associ\'ees aux courbes int\'egrales localement planes, des techniques de cycles proches et \'evanescents, et une interpr\'etation en termes de la th\'eorie de Springer des espaces de modules de Hitchin paraboliques.
\end{abstract}

\maketitle

\setcounter{tocdepth}{1} 

\tableofcontents
\setcounter{section}{-1}

\section{Introduction}

The topology of the Hitchin system has been intensively studied over decades. In \cite{Ngo}, Ng\^o proved the fundamental lemma of the Langlands program by reducing it to the study of the decomposition theorem \cite{BBD} for the Hitchin system; the topological mirror symmetry conjecture of Hausel--Thaddeus \cite{HT0} was proven by comparing the decomposition theorem for Hitchin systems associated with the Langlands dual groups $\mathrm{SL}_n$ and $\mathrm{PGL}_n$ \cite{GWZ, MS_HT}; the decomposition theorem for the Hitchin system further yields a perverse filtration on the singular cohomology of the Hitchin moduli space, whose interaction with the non-abelian Hodge theory is the main theme of the P=W conjecture \cite{dCHM1}, recently proven \cite{MS_PW, HMMS, MSY}; the decomposition theorem for the Hitchin system also appears naturally in enumerative geometry, particularly in the study of BPS invariants \cite{CDP, MS, KK, DHSM}.

Most of these developments relied heavily on topological and sheaf-theoretic approaches. On the other hand, hints from several recent results \cite{GShen, LW, HPL, MSY} suggest that there should be a motivic theory hidden behind the cohomological structures mentioned above. The purpose of this work is to study motivic aspects of the decomposition theorem and the perverse filtration for the Hitchin system. The upshot is that the geometry of the Hitchin system provides us powerful tools to construct algebraic cycles on the moduli of Higgs bundles, lifting many cohomological statements to the motivic level.


\subsection{Hitchin systems}\label{sec0.1}

Throughout, we work over the complex numbers $\BC$; all Chow groups are taken with rational coefficients. 

Let $\Sigma$ be a nonsingular projective curve of genus $g\geq 2$. Let $M_{n,d}$ be the moduli space of stable Higgs bundles with coprime rank $n$ and degree $d$ on $\Sigma$; it carries the structure of an integrable system, known as the \emph{Hitchin system}:
\begin{equation}\label{Hitchin}
f: M_{n,d} \to B,
\end{equation}
which is a proper and surjective morphism to an affine space $B$. We use $d_f$ to denote the relative dimension of (\ref{Hitchin}). By the decomposition theorem of Beilinson, Bernstein, Deligne, and Gabber \cite{BBD}, the derived pushforward $Rf_{*} \BQ_{M_{n,d}}\in D^\mathrm{b}_\mathrm{c}(B)$ admits a decomposition into (shifted) semisimple perverse sheaves:
\begin{equation}\label{DTHiggs}
Rf_{*}\BQ_{M_{n,d}}[\dim M_{n,d} - d_f] \simeq \bigoplus_{i=0}^{2d_f} {^\mathfrak{p}}\CH^i\left(Rf_{*}\BQ_{M_{n,d}} [\dim M_{n,d} - d_f]\right)\left[-i\right] \in D^\mathrm{b}_\mathrm{c}(B)
\end{equation}
where the semisimple summands on the right-hand side satisfy the Lefschetz symmetry
\begin{equation}\label{RHL}
\sigma^{i}: {^\mathfrak{p}}\CH^{d_f-i}\left(Rf_{*}\BQ_{M_{n,d}} [\dim M_{n,d} - d_f]\right) \xrightarrow{~~\simeq~~}{^\mathfrak{p}}\CH^{d_f+i}\left(Rf_{*}\BQ_{M_{n,d}}[\dim M_{n,d} - d_f]\right)
\end{equation}
induced by the cup-product with the $i$-th power of a relative ample class $\sigma$. The decomposition~(\ref{DTHiggs}) is an important invariant for the Hitchin system (\ref{Hitchin}) and has played a crucial role in all the stories mentioned above. For example, recent resolutions of the P=W conjecture show that the (cohomological) perverse filtration
\begin{equation}\label{perv_filtraion}
P_0 H^*(M_{n,d}, \BQ) \subset P_1 H^*(M_{n,d}, \BQ)\subset \cdots \subset  H^*(M_{n,d}, \BQ)
\end{equation}
induced by (\ref{DTHiggs}) matches the double-indexed weight filtration associated with the corresponding character variety via non-abelian Hodge theory:
\begin{equation*}\label{P=W}
P_k H^*(M_{n,d}, \BQ) = W_{2k}H^*(M_{\mathrm{char}}, \BQ).
\end{equation*}
Furthermore, the Lefschetz symmetry (\ref{RHL}) matches the \emph{curious hard Lefschetz} for the character variety \cite{HRV, Mellit}.

Our first result is that both the decomposition (\ref{DTHiggs}) and the filtration (\ref{perv_filtraion}) are motivic in a strong sense. In particular, this confirms the \emph{motivic decomposition conjecture} of Corti--Hanamura~\cite{CH} for the Hitchin system.

\begin{thm}[Motivic decomposition]\label{thm1}
There exists a decomposition of the relative diagonal into orthogonal projectors:
\[
[\Delta_{M_{n,d}/B}] = \sum_{i=0}^{2d_f} \mathfrak{r}_i, \quad\quad \mathfrak{r}_i\circ \mathfrak{r}_i = \mathfrak{r}_i, \quad \mathfrak{r}_i\circ \mathfrak{r}_j =0, \quad i\neq j,\]
which induces a sheaf-theoretic decomposition \eqref{DTHiggs}. Here $[\Delta_{M_{n,d}/B}]$ and $\mathfrak{r}_i$ lie in the Chow group of the relative product $M_{n,d} \times_B M_{n,d}$, and the composition is taken as relative correspondences.
\end{thm}

Our method also yields immediately the relative Lefschetz standard conjecture for the Hitchin system, which shows the algebraicity of the inverse of (\ref{RHL}).

\begin{thm}[Relative Lefschetz]\label{motivic_lefschetz}
    For any relative ample class $\sigma$ and $i > 0$, the inverse of the Lefschetz symmetry \eqref{RHL} is induced by an algebraic cycle $\mathfrak{Z}_{\sigma,i}$ on $M_{n,d} \times_B M_{n,d}$:
    \[
    \mathfrak{Z}_{\sigma,i}: {^\mathfrak{p}}\CH^{d_f+i}\left(Rf_{*}\BQ_{M_{n,d}} [\dim M_{n,d} - d_f]\right) \xrightarrow{~~\simeq~~}{^\mathfrak{p}}\CH^{d_f-i}\left(Rf_{*}\BQ_{M_{n,d}}[\dim M_{n,d} - d_f]\right).
    \]
\end{thm}

Under the language of \cite{CH}, we obtain from Theorem \ref{thm1} a decomposition in the category of relative Chow motives over $B$:
\begin{equation}\label{motivic_DT}
h\left(M_{n,d}\right) = \bigoplus_{i}h_i(M_{n,d}) \in \mathrm{CHM}(B)
\end{equation}
with 
\[
h(M_{n,d})=(M_{n,d},[\Delta_{M_{n,d}/B}],0), \quad h_i(M_{n,d}) = (M_{n,d}, \mathfrak{r}_i,0)
\]
whose homological realization under the Corti--Hanamura functor 
\[
\mathrm{CHM}(B) \to D^\mathrm{b}_\mathrm{c}(B)
\]
recovers (\ref{DTHiggs}). As a consequence, we obtain a motivic perverse filtration
\begin{equation}\label{Motivic_P}
P_{\bullet} h(M_{n,d}) \subset  h(M_{n,d});
\end{equation}
see Section \ref{sec2.1}. This motivic perverse filtration specializes to both the cohomological perverse filtration (\ref{perv_filtraion}) and a perverse filtration on the Chow groups of $M_{n,d}$; the latter has not yet been much studied.

\begin{rmk}\label{rmk0}
    We conclude Section \ref{sec0.1} with some remarks.
    \begin{enumerate}
        \item[(a)] The decomposition (\ref{motivic_DT}) we constructed in Theorem \ref{thm1} specializes to a particular sheaf-theoretic isomorphism (\ref{DTHiggs}). Conversely, it is interesting to explore whether the natural sheaf-theoretic decompositions \cite{D_splitting, dC_splitting} admit motivic liftings. 
        \item[(b)] Motivated by the P=W conjecture, it is conjectured that the Hitchin system admits a multiplicative decomposition \ref{motivic_DT}, \emph{i.e.}, a motivic decomposition compatible with cup-product; we refer to \cite{BMSY,Sh} for detailed discussions. It is not clear whether the decomposition given by Theorem \ref{thm1} is multiplicative. 
        \item[(c)] A natural conjectual enhancement of Theorem \ref{thm2} is the relative Chow--Lefschetz standard conjecture for the Hitchin system, \emph{i.e.}, the Lefschetz symmetry for the motivic decomposition (\ref{motivic_DT}). Our argument does not prove that, and it is an interesting direction to explore.
    \end{enumerate}
\end{rmk}


\subsection{Motivic strong perversity of Chern classes}\label{Sect0.2}

Let $\CU$ be a universal Higgs bundle over~$\Sigma\times M_{n,d}$. After normalization, it induces a canonical Chern character
\[
\widetilde{\mathrm{ch}}(\CU): = \sum_{k}\widetilde{\mathrm{ch}}_k(\CU) \in \mathrm{CH}^*( \Sigma \times M_{n,d} ), \quad \widetilde{\mathrm{ch}}_k(\CU) \in \mathrm{CH}^k( \Sigma \times M_{n,d} )
\]
which is independent of the choice of the universal bundle $\CU$. It was proven in \cite{MS_PW} that the cohomology class $\widetilde{\mathrm{ch}}_k(\CU)$ has strong perversity $k$, \emph{i.e.}, its action on the direct image complex
\[
\widetilde{\mathrm{ch}}_k(\CU): Rf_{*} \BQ_{\Sigma \times M_{n,d}} \to Rf_{*} \BQ_{\Sigma \times M_{n,d}}[2k] 
\]
via cup-product satisfies that
\begin{equation}\label{strong_perversity}
\widetilde{\mathrm{ch}}_k(\CU):{^\mathbf{p}\tau_{\leq i} Rf_* \BQ_{\Sigma \times M_{n,d}}} \to  \left( {^\mathbf{p}\tau_{\leq i+k}} Rf_* \BQ_{\Sigma \times M_{n,d}}\right)  [2k] .
\end{equation}
Here, we also use $f$ to denote the product morphism $\Sigma \times M_{n,d} \rightarrow \Sigma \times B$ and consider the corresponding perverse filtration.
Specializing to global cohomology, the equation (\ref{strong_perversity}) yields
\[
\widetilde{\mathrm{ch}}_k(\CU) : P_iH^*(\Sigma \times M_{n,d}, \BQ) \to P_{i+k}H^{*+2k}(\Sigma \times M_{n,d}, \BQ),
\]
which further implies the P=W conjecture; we refer to \cite[Section 2]{MS_PW} for details on the reduction.

In this paper, we show that the sheaf-theoretic strong perversity (\ref{strong_perversity}) for Chern classes can be lifted motivically using the motivic perverse filtration (\ref{Motivic_P}).

\begin{thm}[Motivic strong perversity of Chern classes]\label{thm2}
We have for any $i,k\in \BZ$:
\[
\widetilde{\mathrm{ch}}_k(\CU): P_ih(\Sigma \times M_{n,d}) \to P_{i+k}h(\Sigma \times M_{n,d})(k) \in \mathrm{CHM}(\Sigma \times B).
\]
Here the action is given by cup-product and $(k)$ stands for the $k$-th Tate twist.
\end{thm}

\subsection{Motivic $\chi$-independence}

Now we consider two degrees $d,d'$ which are both coprime to $n$. The Hitchin systems associated with $M_{n,d}, M_{n,d'}$ share the same base,
\[
f_d: M_{n,d} \to B, \quad f_{d'}: M_{n,d'} \to B.
\]
Here we use $f_\bullet$ to indicate the dependence on the degree.

The following result was predicted by the Hausel--Thaddeus conjecture on the topological mirror symmetry \cite{HT0}, and was first proven by Groechenig--Wyss--Ziegler \cite{GWZ} for perverse sheaves via point counting over finite fields and Maulik--Shen \cite{MS} for Hodge modules using support theorem and the vanishing cycle functor:
\begin{equation}\label{coh_chi}
Rf_{d*} \BQ_{M_{n,d}} \simeq Rf_{d'*} \BQ_{M_{n,d'}}.
\end{equation}
This equation is an instance of the \emph{$\chi$-independence} phenomenon: the moduli space of Higgs bundles can be viewed as the moduli space of $1$-dimensional stable sheaves properly supported on the surface $T^*\Sigma$ in the curve class $n[\Sigma]$, and (\ref{coh_chi}) implies that the decomposition theorem is independent of the choice of the Euler characteristic $\chi\left(= \mathrm{deg} +(1-g)n\right)$ of the sheaves. This phenomenon has tight connections to enumerative geometry \cite{Toda}, and has been generalized to singular moduli spaces using cohomological Donaldson--Thomas theory \cite{MS, KK}. Recently, Hoskins and Pepin Lehalleur \cite{HPL} proved the $\chi$-independence of the Chow motives of the moduli of Higgs bundles:
\begin{equation}\label{chow_chi}
h(M_{n,d}) = h(M_{n,d'}) \in \mathrm{CHM}(\mathrm{pt}).
\end{equation}

We prove a motivic $\chi$-independence result
for the relative Chow motives over $B$, which enhances both the sheaf-theoretic result (\ref{coh_chi}) and the motivic result (\ref{chow_chi}).

\begin{thm}[Motivic $\chi$-independence]\label{thm3}
Assume $(n,d)=(n,d')=1$, we have an isomorphism of relative Chow motives:
\begin{equation}\label{iso_CM}
h(M_{n,d}) \simeq h(M_{n,d'}) \in \mathrm{CHM}(B)
\end{equation}
preserving the motivic perverse filtrations
\[
P_kh(M_{n,d}) \simeq P_kh(M_{n,d'})\in \mathrm{CHM}(B).
\]
\end{thm}

\begin{rmk}
\begin{enumerate}
    \item[(a)] To the best of our knowledge, the isomorphism of the relative Chow motives~(\ref{iso_CM}) (even without the motivic perverse filtrations) was not known before. After pushing to a point, this recovers the isomorphism (\ref{chow_chi}) of Hoskins and Pepin Lehalleur.
    \item[(b)] By Proposition \ref{prop2.2}, we obtain that $f_d: M_{n,d} \to B$ and $f_{d'}: M_{n,d'}\to B$ have isomorphic motivic decompositions; \emph{i.e.}, there exists motivic decompositions
    \[
    h(M_{n,d}) = \bigoplus_i h_i(M_{n,d}), \quad h(M_{M_{n,d'}}) = \bigoplus_i h_i(M_{n,d'})
    \]
    satisfying that \[
    h_i(M_{n,d}) \simeq h_i(M_{n,d'}) \in \mathrm{CHM}(B).
    \]
\end{enumerate}

\end{rmk}

\subsection{Background and relations to other work}
In this section, we discuss briefly some background and relations to some previous work.

\subsubsection{Motivic decomposition conjecture}
Let $f: X \to B$ be a proper morphism of nonsingular varieties. The {motivic decomposition conjecture} of Corti--Hanamura \cite{CH} predicts that in general the sheaf-theoretic decomposition
\[
Rf_*\BQ_X[\dim X - d_f] \simeq \bigoplus_{i=0}^{2d_f} {^\mathfrak{p}}\CH^i\left(Rf_*\BQ_X[\dim X - d_f] \right)[-i], \quad d_f:= \dim X\times_B X - \dim X
\]
admits a \emph{motivic} lifting. More precisely, as in the statement of Theorem \ref{thm1}, they conjectured that there is a decomposition of the relative diagonal cycle into orthogonal projectors
\[
[\Delta_{X/B}] = \sum_{i=0}^{2d_f} \mathfrak{r}_i, \quad\quad \mathfrak{r}_i\circ \mathfrak{r}_i = \mathfrak{r}_i, \quad \mathfrak{r}_i\circ \mathfrak{r}_j =0, \quad i\neq j,
\]
which induces a sheaf-theoretic decomposition as above. In general, the conjectural existence of the motivic decomposition is wide open. We note two major challenges in constructing such a decomposition. First, in the special case when $X$ is a nonsingular projective variety and $B$ is a point, the motivic decomposition conjecture is equivalent to the existence of a Chow--K\"unneth decomposition for $X$. This was conjectured by Murre \cite{Murre}, and has only been verified in very limited cases. Second, even if we are able to find orthogonal projectors over a Zariski dense open subset $U\subset B$ (\emph{e.g.}~the locus where $f$ is smooth), it is unclear if these projectors over $U$ can be extended to orthogonal projectors over $B$. In particular, the cycles given by the Zariski closures of the orthogonal projectors over $U$ are not orthogonal projectors in general, and understanding the behaviour of the ``boundary" is crucial.

Recently there have been some attempts to apply the theory of motivic intermediate extension of Wildeshaus to construct motivic decompositions for certain proper morphisms with \emph{Tate fibers}; we refer to \cite{CDN} and references therein for more details.

\subsubsection{Derived equivalences} \label{Sec0.2}

For a certain class of examples, one approach to the motivic decomposition conjecture is to use symmetries of the derived category of coherent sheaves. It was first found by Beauville \cite{B} and Deninger--Murre \cite{DM} over 35 years ago that descending Mukai's Fourier transform \cite{Mukai} from the derived category of coherent sheaves to algebraic cycles yields a motivic decomposition for an abelian scheme $f: A\to B$. Later, the Fourier transform was further used by K\"unnerman \cite{Kun} to construct a Lefschetz decomposition. In this case the morphism $f$ is smooth and each summand appeared in the decomposition theorem is a (shifted) semisimple local system.

The case of abelian fibrations with singular fibers is more complicated, and the perverse $t$-structure comes into play. On the derived category side, Arinkin \cite{A2} extended the Fourier--Mukai transform to compactified Jacobian fibrations associated with a family of projective integral locally planar curves. Arinkin's work was motivated by an attempt to understanding the classical limit of the geometric Langlands correspondence for Hitchin systems \cite[Section~4.5]{A2}. Recently, motivated by the P=W conjecture for Hitchin moduli spaces \cite{dCHM1, dCMS,MS_PW, HMMS}, a theory of Fourier transform for relative Chow motives was established in \cite{MSY} using the Arinkin sheaf. As a consequence of the Fourier theory and Ng\^o's support theorem \cite{Ngo}, it was obtained in \cite{MSY} that the motivic decomposition conjecture holds for compactified Jacobians fibrations associated with projective integral locally planar curves. 
We note that, for the singular fibers, this approach matches the complexity of the Arinkin sheaf with the complexity of intermediate extensions of the local systems.

\subsubsection{Hitchin systems}
For the Hitchin system (\ref{Hitchin}), the techniques described above imply the motivic decomposition conjecture for the Hitchin system over a Zariski open subset $B^{\mathrm{ell}} \subset B$, called \emph{the elliptic locus}, formed by integral spectral curves. This open subset is exactly the locus over which the Hitchin fibers are integral. However, to the best of our knowledge, it is not yet known how to extend Arinkin's derived equivalence \cite{A2} from the elliptic locus to the total Hitchin base $B$. Even worse, the full support theorem \emph{fails} for (\ref{Hitchin}); the results of~\cite{dCHeM} suggest that the supports of the decomposition theorem associated with (\ref{Hitchin}) are very complicated and there is no complete description by far. In view of the approach of \cite{ACLS}, the lack of the control of the supports provides also the difficulty in proving the relative Lefschetz standard conjecture. The main point of this paper is to give a proof of these conjectures for the Hitchin system~(\ref{Hitchin}) in the absence of the aforementioned geometric ingredients.

\subsubsection{Idea of the proof of Theorem \ref{thm1}}

The strategy for overcoming these difficulties is to interpret the Hitchin moduli space as a global analogue of a Lie algebra, thereby enabling the use of tools from Springer theory. Rather than extending the motivic decomposition from~$B^{\mathrm{ell}}$, we instead consider alternative Hitchin moduli spaces associated with certain parabolic structures, where all the Hitchin fibers are integral. The correspondences arising from Springer theory yield algebraic cycles that relate $M_{n,d}$ to the parabolic moduli space. Consequently, the motivic decomposition for $f: M_{n,d} \to B$ can be derived from that of the parabolic Hitchin system. Finally, to verify that the algebraic cycles we construct yields the desired homological realization, we combine techniques of nearby and vanishing cycles with Springer theory. As a byproduct, we show that Corti--Hanamura's motivic decomposition conjecture is preserved under specialization; see Corollary \ref{cor4.2}.

Springer theory has been applied to the study of moduli spaces of Higgs bundles and character varieties in connection with various questions and conjectures; see \cite{DGT, Mellit, MS_PW, HMMS}.

\subsection{Acknowledgements}

J.S.~gratefully acknowledges the hospitality of both the math department of MIT and the Isaac Newton Institute at Cambridge during his stay in the spring of 2024 where part of this work was completed.

D.M.~was supported by a Simons Investigator Grant.
J.S.~was supported by the NSF grant DMS-2301474, a Sloan Research Fellowship, and a Simons Fellowship during his visit at the Isaac Newton Institute.

\section{Moduli of Higgs bundles}\label{Section1}

In this section, we review the geometry of the moduli of Higgs bundles. Throughout, we fix a nonsingular projective curve $\Sigma$ of genus $g\geq 2$, and fix two integers $n, d$ with $(n,d)=1$.


\subsection{Moduli spaces}

We denote by $M_{n,d}$ the moduli space of stable Higgs bundles 
\[
(\CE, \theta), \quad \theta: \CE \to \CE\otimes \omega_\Sigma, \quad \mathrm{rk}(\CE) = n, \quad \mathrm{deg}(\CE) = d,
\]
where the stability condition is with respect to the slope $\mu(\CE,\theta) = \mathrm{deg}(\CE)/\mathrm{rk}(\CE)$. The moduli space admits a proper morphism to an affine space
\begin{equation}\label{map_f}
f: M_{n,d} \to B:= \bigoplus_{i=1}^n H^0\left(\Sigma, \omega_\Sigma^{\otimes i} \right).
\end{equation}
This map is given by calculating the characteristic polynomial of the Higgs field $\theta$:
\[
f(\CE, \theta) = (\mathrm{tr}(\theta), \mathrm{tr}(\wedge^2 \theta), \cdots, \mathrm{det}(\theta)) \in B,
\]
and is known as the \emph{Hitchin system}. The moduli space $M_{n,d}$ is a holomorphic symplectic variety and the Hitchin system $f$ is a Lagrangian fibration \cite{Hit, Hit1}. Later, when we need to consider the moduli spaces $M_{n,d}$ with different $d$, we will write $f_d$ to indicate the dependence of the Hitchin system on the degree.

By the Beauville--Narasimhan--Ramanan correspondence \cite{BNR}, the Hitchin base $B$ can be viewed as the parameter space of spectral curves; these are properly supported curves in the open surface $\mathrm{Tot}_\Sigma(\omega_\Sigma)$ that are degree $n$ covers of the $0$-section $\Sigma$. We define the elliptic locus~$B^{\mathrm{ell}} \subset B$ to be the open subset formed by integral spectral curves. The results of \cite[Theorem 0.3 and Section 4]{MSY} thus establish the motivic decomposition for the restricted Hitchin system
\[
f^{\mathrm{ell}}: M^{\mathrm{ell}}_{n,d} \to B^{\mathrm{ell}}.
\]

As we commented in Section \ref{Sec0.2}, at the level of algebraic cycles it is not easy to extend orthogonal projectors from $B^{\mathrm{ell}}$ to $B$. A plausible approach is to extend first Arinkin's Fourier--Mukai transform from $M^{\mathrm{ell}}_{n,d}$ to the total space $M_{n,d}$, and then extend the Fourier theory of relative Chow motives of \cite{MSY} as well. Even if we are able to achieve this, a potential difficulty is the complexity of the supports (see Remark \ref{remark1.1} below) which makes it hard to calculate the homological realization of the projectors. 

Instead, our approach in this paper is to first construct orthogonal projectors for certain parabolic Hitchin moduli space which we review in the following sections; then we show that these projectors induce the desired orthogonal projectors of the original Hitchin system. This path of reduction using parabolic Hitchin moduli spaces was applied in proving the P=W conjecture: it was first used in the proof of Hausel--Mellit--Minets--Schiffmann \cite{HMMS}, and was later also used in the proof of Maulik--Shen--Yin \cite[Section 5.4]{MSY}.

\begin{rmk}\label{remark1.1}
We can also consider the moduli space of stable $\mathfrak{D}$-twisted Higgs bundles 
\[
(\CE, \theta), \quad \theta: \CE \to \CE \otimes \omega_\Sigma(\mathfrak{D})
\]
as in \cite{CL} with $\mathfrak{D} \geq 0$ an effective divisor. The case of $\mathfrak{D}=0$ recovers the Hitchin system~(\ref{map_f}) we considered above. We note that the decomposition theorem for the Hitchin system behaves very differently when $\mathfrak{D}=0$ and $\mathfrak{D}>0$. For $\mathfrak{D}>0$, Chaudouard--Laumon \cite{CL} showed that every simple summand that appeared in the decomposition theorem has full support $B$. This full support property has been generalized to singular moduli spaces without coprime assumption on $n$ and $d$ \cite{MS} and has many applications \cite{MS_HT, GShen, HPL, MS_PW}. However, when $D=0$, the supports for the Hitchin system are very different. In \cite{dCHeM}, the authors classified all the supports over an open subset of the Hitchin base formed by \emph{reduced} spectral curves. However it is not well understood whether there are supports outside this open subset. A question raised by Mauri--Migliorini \cite[Question 1.5]{MM} asks if the full support property holds for the singular Hitchin moduli space of rank $n$ and degree $d=0$. In Section \ref{sec4.3.5}, we will discuss briefly how to extend the results of this paper to the $\mathfrak{D}$-twisted case.
\end{rmk}

\subsection{Parabolic Higgs bundles}\label{Sec1.2}

In this section, we review the parabolic Hitchin moduli space (\emph{c.f.}~\cite[Section 8.4]{HMMS}). Our purpose is to describe a smooth family of moduli spaces, connecting the moduli of strongly parabolic Higgs bundles with nilpotent residue (which has complicated spectral curves) to a certain moduli space of parabolic Higgs bundles with general residue (which only has integral spectral curves).

We fix a point $p\in \Sigma$. Denote by ${M}^{\mathrm{par}}$ the moduli space of stable parabolic Higgs bundles~$(\CE, \theta, F_p)$, where $(\CE,\theta)$ is a meromorphic Higgs bundle
\[
(\CE, \theta), \quad \theta: \CE \to \CE\otimes \omega_\Sigma(p), \quad \mathrm{rk}(\CE) = n, \quad \mathrm{deg}(\CE) = d,
\]
and $F_p$ is a complete flag of $\CE$ over the point $p$ such that the residue $\mathrm{res}_p(\theta)$ of the Higgs field over the point $p$ preserves the flag $F_p$. The stability condition here requires that any parabolic Higgs subbundle has a smaller slope. We have the corresponding Hitchin map which is proper over an affine space
\[
f^{\mathrm{par}}: M^{\mathrm{par}} \to B^{\mathrm{par}} := \bigoplus_{i = 1}^nH^0\left(\Sigma, \omega_\Sigma(p)^{\otimes i}\right).
\]

We record the following lemma which is a direct consequence of the Riemann--Roch formula.

\begin{lem}\label{lem1.20}
We have 
\[
 H^0(\Sigma, \omega_\Sigma) = H^0(\Sigma, \omega_\Sigma(p)).
\]
\end{lem}

Since $\mathrm{res}_p(\theta)$ preserves a complete flag, its eigenvalues naturally lie in an affine space 
\[
\BA^n:=\left\{(\lambda_1, \cdots, \lambda_n) \mid \lambda_i \in \BA^1 \right\}, 
\]
and Lemma \ref{lem1.20} ensures that the sum of the eigenvalues is always $0$; in other words, the parameter space of (ordered) eigenvalues of $\mathrm{res}_p(\theta)$ lies in a codimension $1$ hyperplane
\[
\CA=\left\{(\lambda_1, \cdots, \lambda_n) \mid \sum_i\lambda_i = 0\right\}\subset \BA^{n}.
\]
As a result, the moduli space admits a natural morphism
\begin{equation}\label{chi_p}
\chi_p: {M}^{\mathrm{par}} \to \CA \simeq \BA^{n-1}
\end{equation}
which is in fact smooth; see e.g.~\cite[Corollary 2.11]{LL}. Each closed fiber is a moduli of stable parabolic Higgs bundles with prescribed eigenvalues at $p$; this is a nonsingular symplectic variety.

On the other hand, the Hitchin base $B^\mathrm{par}$ only remembers the \emph{unordered} eigenvalues of~$\mathrm{res}_p(\theta)$ through its quotient by
\[
B^{\mathrm{par}}_0 = \bigoplus_{i = 1}^nH^0\left(\Sigma, \omega_\Sigma^{\otimes i}((i - 1)p)\right) \subset B^{\mathrm{par}},
\]
see \cite[Sections 3 and 4]{LM}. There is a natural identification
\[
\mathcal{A}\sslash \mathfrak{S}_n \simeq \overline{B}^{\mathrm{par}} := B^{\mathrm{par}}/B^{\mathrm{par}}_0.
\]
We consider the \emph{ordered} Hitchin map
\[
h: M^{\mathrm{par}} \to W:= \mathcal{A} \times_{\overline{B}^{\mathrm{par}}} B^{\mathrm{par}}
\]
which can be thought of as a family of fibrations over $\mathcal{A}$:
\[
\begin{tikzcd}[column sep=small]
M^{\mathrm{par}} \arrow[rr, "h"] \arrow[rd, ""] & & W \arrow[ld, ""] \\
& \mathcal{A}.
\end{tikzcd}
\]

For our purposes, we take a general line $T \simeq \BA^1 \subset \CA$ passing through the origin and restrict the family (\ref{chi_p}) over $T$:
\[
\chi_{p}: M^{\mathrm{par}}_T \to T.
\]
The geometry of this family is described by \cite[Sections 8.3 and 8.4]{HMMS} and \cite[Section 5.4.2]{MSY}, connecting two types of Hitchin moduli spaces. We review in the following some details which are relevant for our purposes.

\begin{enumerate}
    \item[(a)] Let $\eta \in T$ be a general point which represents $n$ \emph{distinct} eigenvalues satisfying the condition that no proper subset of the eigenvalues has sum $0$. Let $C_\eta \to W_\eta$ be the universal spectral curve with fixed eigenvalues $\eta$ at $p$. By the generality of $\eta$ and Lemma \ref{lem1.20}, every curve in the family $C_\eta \to W_\eta$ is integral. The fiber $M^{\mathrm{par}}_\eta: = \chi_p^{-1}(\eta)$ therefore admits a compactified Jacobian fibration
    \[
    h_\eta: M^{\mathrm{par}}_\eta \to W_\eta.
    \]
    Let $T^\circ \subset T$ be the Zariski open subset such that the spectral curves in $W_\eta$ are integral for any $\eta \in T^\circ$. We thus obtain a family of compactified Jacobian fibrations over $T^\circ$:
\[
\begin{tikzcd}[column sep=small]
M^{\mathrm{par}}_{T^\circ} \arrow[rr, "h^\circ"] \arrow[rd, ""] & & W_{T^\circ} \arrow[ld, ""] \\
& T^\circ.
\end{tikzcd}
\]

    \item[(b)] The central fiber $M^{\mathrm{par}}_0:= \chi_p^{-1}(0)$ together with its associated Hitchin system
    \[
h_0:    M^{\mathrm{par}}_0 \to W_0
    \]
is a classical object, known as the moduli of stable strongly parabolic Higgs bundles. The variety $M^{\mathrm{par}}_0$ parameterizes triples $(\CE, \theta, F_p)$ with similar condition as for stable parabolic Higgs bundle, but the main change is that the residue of the Higgs field preserves the flag in a stronger sense: \[
\mathrm{res}_p(\theta): F_{p,j} \to F_{p,{j+1}}.
\]
In other words, $M_0^{\mathrm{par}}$ is the moduli of stable parabolic Higgs bundles with nilpotent residue.
\end{enumerate}

In conclusion, by specializing over $T$, we are able to connect the geometry of a compactified Jacobian fibration to the Hitchin system associated with the moduli of strongly parabolic Higgs bundles:
\begin{equation}\label{sp_diagram}
    h_\eta: M^{\mathrm{par}}_\eta \to W_\eta
\quad \rightsquigarrow \quad h_0:    M^{\mathrm{par}}_0 \to W_0.
\end{equation}
We now connect the latter to the more classical Hitchin system~$f: M_{n,d} \to B$ via a correspondence following \cite[Section 8.6]{HMMS}.

\subsection{Correspondences}\label{sec1.3}
We consider the closed subvariety parameterizing parabolic Higgs bundles with $0$-residue at $p$,
\begin{equation}\label{1.3_1}
\widetilde{M}_0=\{(\CE, \theta, F_p) \mid \mathrm{res}_p(\theta) = 0\} \subset M_0^{\mathrm{par}} = \{(\CE, \theta, F_p) \mid \mathrm{res}_p(\theta) \textrm{ is nilpotent}\}.
\end{equation}
In particular, every parabolic Higgs bundle in $\widetilde{M}_0$ has no pole at $p$. Hence there is a  well-defined forgetful map:
\begin{equation}\label{1.3_2}
\widetilde{M}_0 \to M_{n,d}, \quad (\CE, \theta, F_p) \mapsto (\CE, \theta)
\end{equation}
where every closed fiber is isomorphic to a flag variety $\mathrm{Fl}_n$ parameterizing all possible choices of flags.

We summarize the discussion above in the diagram
\begin{equation}\label{corr}
    \begin{tikzcd}
\widetilde{M}_0 \arrow[r, "\iota"] \arrow[d, "\pi"]
&M_0^{\mathrm{par}} \\
M_{n,d} 
\end{tikzcd}
\end{equation}
where $\iota$ is the closed embedding (\ref{1.3_1}) and $\pi$ is the smooth projection (\ref{1.3_2}). All the three moduli spaces in (\ref{corr}) map to a common base $W_0$ with all the induced natural diagrams commuting: the map from $M_0^{\mathrm{par}}$ is the parabolic Hitchin map $h_0$, the map from $\widetilde{M}_0$ is the induced map, and the map from $M_{n,d}$ is the Hitchin map $f: M_{n,d} \to B$ composed with the natural inclusion of the linear subspace $B \subset W_0$.


Our purpose is to construct certain algebraic cycles on $M_{n,d} \times_B M_{n,d}$ which serve as orthogonal projectors. The diagram (\ref{corr}) allows us to construct cycles for $M_{n,d}$ from $M_0^{\mathrm{par}}$, \emph{i.e.}, we have
\begin{equation} \label{eq:defGamma}
\Gamma: \mathrm{CH}_*({M^{\mathrm{par}}_0 \times_{W_0}M^{\mathrm{par}}_0}) \to \mathrm{CH}_*( M_{n,d} \times_B M_{n,d} )
\end{equation}
given by the formula
\[
\Gamma(\alpha):= {^\mathfrak{t}[\widetilde{M}_0]}\circ \alpha \circ [\widetilde{M}_0].
\]
Here we view~$[\widetilde{M}_0]$ to be a correspondence from $M_{n,d}$ to $M^{\mathrm{par}}_0$ via (\ref{corr}) and $ {^\mathfrak{t}[\widetilde{M}_0]}$ to be its transpose. The change of indices of the Chow groups through $\Gamma$ is given by the dimension formula:
\[
\mathrm{codim}(\iota) = \mathrm{rel.dim}(\pi) = \binom{n}{2};
\]
indeed, each closed fibers of $\pi$ is a flag variety whose dimension is $\binom{n}{2}$. In particular, $\Gamma$ sends $\mathrm{CH}_{\dim M_0^{\mathrm{par}}}({M^{\mathrm{par}}_0 \times_{W_0}M^{\mathrm{par}}_0})$ to $\mathrm{CH}_{\dim M_{n,d}}( M_{n,d} \times_B M_{n,d} )$.

The following lemma gives a sufficient condition for $\Gamma(\alpha)$ to vanish.  

\begin{lem}\label{lem1.2}

Suppose we are given a flat, proper surjective morphism $p: Z \to W_0$ with $Z$ irreducible and nonsingular and that there exists a factorization of $\alpha$
\[
\alpha := \alpha_2 \circ \alpha_1 \in \mathrm{CH}_{\dim  M_0^{\mathrm{par}}}({M^{\mathrm{par}}_0 \times_{W_0}M^{\mathrm{par}}_0})
\]
such that either
\[
\alpha_1 \in \mathrm{CH}_{>\dim (M_0^{\mathrm{par}}\times_{W_0}Z)-\binom{n}{2}}(M_0^{\mathrm{par}}\times_{W_0}Z) \textup{ or } \alpha_2 \in \mathrm{CH}_{>\dim (M_0^{\mathrm{par}}\times_{W_0}Z)-\binom{n}{2}}(Z\times_{W_0}M_0^{\mathrm{par}}).\]
Then we have
\[
\Gamma(\alpha) = 0 \in  \mathrm{CH}_{\dim M_{n,d}}( M_{n,d} \times_{W_0} M_{n,d}) =\mathrm{CH}_{\dim  M_{n,d}}( M_{n,d} \times_{B} M_{n,d}).
\]
\end{lem}
\begin{proof}
We prove the vanishing under the first hypothesis; the argument under the second hypothesis is identical.  If we consider the composition $\alpha_1 \circ [\widetilde{M}_0] \in \mathrm{CH}_*(M_{n,d}\times_{W_0}Z)$, 
the assumption on $\alpha_1$
implies that this lives in degree strictly larger than
$$
\dim (M_0^{\mathrm{par}}\times_{W_0}Z)-2\binom{n}{2}.
$$
The assumption on $p$ implies that this bound equals
\[
\dim M_0^{\mathrm{par}}  + \dim Z - \dim W_0-2\binom{n}{2} = \dim M_{n,d} + \dim Z - \dim W_0
\]
which is the dimension of $M_{n,d}\times_{W_0}Z$.  As a result, we have the vanishing of $\alpha_1 \circ [\widetilde{M}_0]$ so $\Gamma(\alpha)$ vanishes as well.
\end{proof}

We conclude this section by giving an overview of the rest of the paper. Our strategy is to first construct projectors for compactified Jacobian fibrations and show that these algebraic cycles satisfy the desired properties as in Theorems \ref{thm2} and \ref{thm3} using the techniques of \cite{MSY}. We will discuss this in Section \ref{Section2}, which solves the problems for $h_\eta: M^{\mathrm{par}}_\eta \to W_\eta$ of Section~\ref{Sec1.2}(a). Then applying the specialization and the correspondence $\Gamma$, we obtain algebraic cycles for~$M_{n,d}$; we further apply a Springer theory argument to show that these cycles form orthogonal projectors as desired. This is treated in Section \ref{Section3}. Finally, in Section \ref{Section4}, we show that the projectors obtained via $\Gamma$ have the correct homological realization under the Corti--Hanamura functor.

\section{Projectors for compactified Jacobians}\label{Section2}

In this section, we first introduce the notion of {motivic perverse filtrations}. Then we treat the case of compactified Jacobians associated with locally planar curves, where our main tool is the Fourier theory developed in \cite{MSY}. For our purpose, we need a slight extension of the results of \cite{MSY} in Section \ref{Section2.2}. Then we apply these results to the parabolic Hitchin system of Section \ref{Sec1.2}(a), which further allows us to prove analogues of Theorems \ref{thm1}, \ref{thm2}, and \ref{thm3} for~$h_0: M^{\mathrm{par}}_0 \to W_0$.

\subsection{Motivic background}\label{sec2.1}
We begin by reviewing the motivic background necessary for our main results, and by introducing the notion of a \emph{motivic perverse filtration}.

\subsubsection{Relative Chow motives}
The theory of relative Chow motives of Corti--Hanamura \cite{CH} is built to be compatible with the decomposition theorem, adapts well to non-proper bases, and admits natural Chow/homological realizations~\cite{GHM}. We briefly recall the theory below, and refer to \cite[Section 2.2]{MSY} for more details.

We work over a nonsingular base variety $B$. Recall that the group of degree $k$ relative correspondences between two nonsingular proper $B$-schemes $X, Y$ with $Y$ irreducible is
\[
\mathrm{Corr}^k_B(X, Y) := \mathrm{CH}_{\dim Y - k}(X \times_B Y).
\]
For reducible $Y$, we decompose $Y = \sqcup_\alpha Y_\alpha$ and define the group of correspondences accordingly. Compositions of relative correspondences are defined via refined intersection theory. The category of relative Chow motives $\mathrm{CHM}(B)$ consists of objects triples $(X, \mathfrak{p}, m)$ where $X$ is a nonsingular proper $B$-scheme, $\mathfrak{p} \in \Corr^0_B(X, X)$ is a projector, and $m \in \mathbb{Z}$. For example, the motive of $X$ is
\[
h(X) := (X, [\Delta_{X/B}], 0)
\]
where $\Delta_{X/B}$ is the relative diagonal. Morphisms between two motives $M = (X, \mathfrak{p}, m)$, \mbox{$N = (Y, \mathfrak{q}, n)$} are
\[
\Hom_{\mathrm{CHM}(B)}(M, N) := \mathfrak{q} \circ \Corr^{n - m}_B(X, Y) \circ \mathfrak{p}.
\]
Note that in \cite[Section 2.3.3]{MSY}, we also introduced relative binary correspondences
\[
\Corr_B^k(X, Y; Z) := \mathrm{CH}_{\dim Z - k}(X \times_B Y \times_B Z)
\]
where $X, Y, Z$ are nonsingular proper $B$-schemes. Binary morphisms from $M = (X, \Fp, m)$, $N = (Y, \Fq, n)$ to $P = (Z, \Fr, p)$ are given by
\[
\Fr \circ \Corr_B^{p - m - n}(X, Y; Z) \circ (\Fp \times \Fq).
\]
For example, the motivic cup-product
\[
\cup : h(X) \times h(X) \to h(X)
\]
is induced by the class of the relative small diagonal
\[
[\Delta^{\mathrm{sm}}_{X/B}] \in \Corr_B^0(X, X; X);
\]
see \cite[Section 2.3.3]{MSY} for more details. Finally, the degree $k$ Chow group of $M = (X, \mathfrak{p}, m)$ is
\[
\mathrm{CH}^k(M) := \Hom_{\mathrm{CHM}(B)}(h(B/B), M(k))
\]
where $M(k) := (X, \mathfrak{p}, m + k)$ is the $k$-th Tate twist of $M$.

\subsubsection{Motivic perverse filtrations}

There is a realization functor
\begin{equation}\label{CH0}
\mathrm{CHM}(B) \to D^\mathrm{b}_\mathrm{c}(B)
\end{equation}
which sends the motive $h(X)$ for $f: X \to B$ to the direct image complex $Rf_*\BQ_X$. The latter admits a natural (sheaf-theoretic) perverse filtration induced by the perverse truncation functor. We introduce the notion of a \emph{motivic perverse filtration} lifting the perverse truncation of $Rf_*\BQ_X \in D^b_c(X)$ to the object $h(X) \in \mathrm{CHM}(B)$. Then we prove that, although a motivic perverse filtration is \emph{a priori} weaker than a motivic decomposition in the sense of Corti--Hanamura \cite{CH}, they are in fact equivalent.

Let $f: X\to B$ be a proper morphism with $X$ irreducible and nonsingular, and let $d_f$ be the defect of semismallness $\dim X\times_B X -\dim X $.

\begin{defn}\label{def_motivic_P}
We say that a sequence of motives
\begin{equation*}
P_k h(X) = (X, \mathfrak{p}_k, 0), \quad 0\leq k \leq 2d_f,
\end{equation*}
which are all summands of $h(X)$, form a \emph{motivic perverse filtration} for $f$, if 
\begin{enumerate}
    \item[(a)](Termination) $P_{2d_f}h(X) = h(X)$;
    \item[(b)](Realization) under the Corti--Hanamura functor (\ref{CH0}) the natural inclusion 
    \[
    P_\bullet h(X)\hookrightarrow h(X)
    \]
    specializes to the natural inclusion induced by the perverse truncation functor
    \[
    ^{\mathfrak{p}}\tau_{\leq \bullet+ (\dim X - d_f)} Rf_*\BQ_X \hookrightarrow  Rf_*\BQ_X;
    \]
    \item[(c)](Semi-orthogonality) for any $k$ we have the relation 
    \[
    \mathfrak{p}_{k + 1} \circ \mathfrak{p}_k = \mathfrak{p}_k.
    \]
\end{enumerate}
\end{defn}

We note that (c) together with an easy induction implies
\begin{equation} \label{eq:semiorth}
\mathfrak{p}_l \circ \mathfrak{p}_k = \Fp_k, \quad k < l.
\end{equation}
Indeed, we have
\[
\mathfrak{p}_l \circ \mathfrak{p}_k = \mathfrak{p}_l \circ \mathfrak{p}_{k + 1} \circ \mathfrak{p}_k = \mathfrak{p}_l \circ \mathfrak{p}_{l - 1} \circ \cdots \circ \mathfrak{p}_{k + 1} \circ \mathfrak{p}_k = \mathfrak{p}_{l - 1} \circ \cdots \circ \mathfrak{p}_{k + 1} \circ \mathfrak{p}_k = \mathfrak{p}_k.
\]
On the other hand, this does not mean that $P_kh(X)$ is a summand of $P_lh(X)$ if $k<l$. So if we define directly
\[
\mathfrak{r}_k: = \mathfrak{p}_k - \mathfrak{p}_{k-1},
\]
they are only semi-orthogonal:
\[
\Fr_{l} \circ \Fr_{k} = 0, \quad k<l,
\]
and we lose control of the other half of the orthogonality
\[
\Fr_{l} \circ \Fr_{k} \stackrel{?}{=} 0, \quad k \geq l.
\]
In particular, the $\mathfrak{r}_k$ are not necessarily projectors. However, the following proposition in \cite{MSY} shows that the existence of a motivic perverse filtration is equivalent to the existence of a motivic decomposition.

\begin{prop}[Motivic perverse filtration $=$ motivic decomposition]\label{prop2.1}
A motivic perverse filtration induces a motivic decomposition
\[
h(X) = \bigoplus_{i=0}^{2d_f} h_i(X) \in \mathrm{CHM}(B).
\]
\end{prop}

\begin{proof}
The proof was given in the \cite[Section 2.5.5]{MSY}; we include it here for completeness. Assume that the motives $P_kh(X) = (X, \mathfrak{p}_k, 0)$ induce a motivic perverse filtration. We consider the modification of projectors:
\[
\widetilde{\Fp}_k := \Fp_k \circ \cdots \circ \Fp_{2d_f}, \quad \overline{\Fp}_k := \widetilde{\Fp}_k - \widetilde{\Fp}_{k - 1}.
\]
Using \eqref{eq:semiorth}, it is straightforward to verify that both $\widetilde{\Fp}_k$ and $\overline{\Fp}_k$ are projectors satisfying the relations
\[
\widetilde{\Fp}_l \circ \widetilde{\Fp}_k = \widetilde{\Fp}_k \circ \widetilde{\Fp}_l = \widetilde{\Fp}_k, \quad k <  l.
\]
In particular, we have
\begin{equation*}\label{p2.1_1} 
\widetilde{P}_kh(X) = \widetilde{P}_{k - 1}h(X) \oplus h_k(X) \in \mathrm{CHM}(B)
\end{equation*}
with
\[
\widetilde{P}_kh(X): = (X, \widetilde{\Fp}_k, 0), \quad h_k(X) := (X,\overline{\mathfrak{p}}_k, 0).
\]
We see from (b) that both motives $\widetilde{P}_kh(X)$ and $P_kh(X)$ have the same realization under~(\ref{CH0}), which is the one given by the perverse truncation functor. In particular, the homological realization of the projector~$\overline{\mathfrak{p}}_k$ is taking the corresponding perverse cohomology of the complex~$Rf_*\BQ_X$. This finishes the proof of the proposition.
\end{proof}

In the following, we further extend the equivalence between motivic perverse filtrations and motivic decompositions of Proposition \ref{prop2.1} to isomorphism classes.

Assume that 
\[
f: X\to B, \quad f': X'\to B
\]
admit motivic perverse filtrations
\begin{equation}\label{motivic_P0}
P_k h(X) = (X, \mathfrak{p}_k, 0), \quad P_k h(X') = (X', \mathfrak{p}'_k, 0), \quad 0\leq k \leq 2d_f=2d_{f'}.
\end{equation}

\begin{defn}
We say that an isomorphism of relative Chow motives
\[
\mathfrak{C}: h(X) \xrightarrow{~~\simeq~~} h(X'),\quad \mathfrak{C}^{-1}: h(X') \xrightarrow{~~\simeq~~} h(X)
\]
with
\[
\mathfrak{C}\in \mathrm{Corr}_B^0(X, X'),\quad \mathfrak{C}^{-1} \in \mathrm{Corr}_B^0(X', X)
\]
preserves the motivic perverse filtrations (\ref{motivic_P0}), if for any $k$ we have
\begin{equation}\label{key_assumption}
    \mathfrak{q}'_{k + 1} \circ \mathfrak{C}\circ \mathfrak{p}_k =0, \quad
     \mathfrak{q}_{k + 1} \circ \mathfrak{C}^{-1}\circ \mathfrak{p}'_k =0.
\end{equation}
Here
\[
\mathfrak{q}_{k + 1}:= [\Delta_{X/B}] - \mathfrak{p}_k, \quad \mathfrak{q}'_{k + 1}:= [\Delta_{X'/B}] - \mathfrak{p}'_k.
\]
\end{defn}

The following proposition shows that, for two relative Chow motives, an isomorphism preserving motivic filtrations automatically yields isomorphic motivic decompositions as well.

\begin{prop}\label{prop2.2}
With the notation as above, assume that there is an isomorphism between the relative Chow motives $h(X)$ and $h(X')$ preserving the motivic perverse filtrations \eqref{motivic_P0}. Then the following holds:
\begin{enumerate}
    \item[(a)]  For any $k$ we have an isomorphism
    \[
    P_kh(X) \simeq P_{k}h(X') \in \mathrm{CHM}(B)
    \]
    compatible with the inclusions $P_kh(X) \hookrightarrow h(X)$ and $P_kh(X') \hookrightarrow h(X')$.
    \item[(b)] The motivic decompositions of $h(X), h(X')$ given by Proposition \ref{prop2.1} are isomorphic, \emph{i.e.}, for any $k$ we have an isomorphism
\[ h_k(X) \simeq h_k(X') \in \mathrm{CHM}(B).
\]
\end{enumerate}
\end{prop}

\begin{proof}
We first prove (a). Assume that the correspondences $\mathfrak{C}, \mathfrak{C}^{-1}$ induce an isomorphism between $h(X), h(X')$ preserving the perverse filtrations.  We claim that the morphisms between the relative Chow motives
\[
\mathfrak{p}'_k\circ \mathfrak{C}\circ \mathfrak{p}_k :P_kh(X) \to P_kh(X'), \quad  \mathfrak{p}_k\circ \mathfrak{C}^{-1}\circ \mathfrak{p}'_k :P_kh(X') \to P_kh(X)
\]
are inverse to each other. In fact, we have 
\begin{align*}
\mathfrak{p}_k = \mathfrak{p}_k \circ \mathfrak{C}^{-1}\circ \mathfrak{C}\circ \mathfrak{p}_k & = \mathfrak{p}_k \circ \mathfrak{C}^{-1}\circ (\mathfrak{p}'_k+ \mathfrak{q}'_{k+1}) \circ \mathfrak{C}\circ \mathfrak{p}_k\\
&= \mathfrak{p}_k \circ \mathfrak{C}^{-1}\circ \mathfrak{p}'_k \circ \mathfrak{C}\circ \mathfrak{p}_k
\end{align*}
where we used $\mathfrak{q}'_{k+1} \circ \mathfrak{C}\circ \mathfrak{p}_k =0$ in the last equation. This proves
\[
 \mathfrak{p}_k\circ \mathfrak{C}^{-1}\circ \mathfrak{p}'_k\circ \mathfrak{C}\circ \mathfrak{p}_k = \mathfrak{p}_k;
\]
the other identity is parallel. 
To show that this isomorphism is compatible with $\mathfrak{C}$ and the inclusions into $h(X)$ and $h(X')$, it suffices to show that
\[
\mathfrak{C}\circ \mathfrak{p}_k = \mathfrak{p}'_k\circ( \mathfrak{p}'_k\circ \mathfrak{C}\circ \mathfrak{p}_k)
\]
which again follows via the condition that $\mathfrak{C}$ preserves the motivic filtration.  We have completed the proof of (a).

Now we prove (b); we follow the notation in the proof of Proposition \ref{prop2.1}. As in (a), it suffices to show that
\begin{equation}\label{8terms}
\left(\widetilde{\mathfrak{p}}_k - \widetilde{\mathfrak{p}}_{k-1}\right)\circ \mathfrak{C}^{-1}\circ \left(\widetilde{\mathfrak{p}}'_k - \widetilde{\mathfrak{p}}'_{k-1}\right)\circ \mathfrak{C}\circ \left(\widetilde{\mathfrak{p}}_k - \widetilde{\mathfrak{p}}_{k-1}\right) = \widetilde{\mathfrak{p}}_k - \widetilde{\mathfrak{p}}_{k-1};
\end{equation}
the other identity follows from a parallel argument.
Expanding the left-hand side of (\ref{8terms}), we obtain 8 terms, each of which is of the form
\begin{equation*}
\pm ~ \widetilde{\mathfrak{p}}_{i_1}\circ \mathfrak{C}^{-1}\circ \widetilde{\mathfrak{p}}'_{i_2}\circ \mathfrak{C}\circ \widetilde{\mathfrak{p}}_{i_3}, \quad i_1, i_2, i_3\in \{k-1,k\}.
\end{equation*}
We claim that 
\begin{equation}\label{each_term}
\widetilde{\mathfrak{p}}_{i_1}\circ \mathfrak{C}^{-1}\circ \widetilde{\mathfrak{p}}'_{i_2}\circ \mathfrak{C}\circ \widetilde{\mathfrak{p}}_{i_3} = \widetilde{\mathfrak{p}}_{\mathrm{min}\{i_1, i_2, i_3\}}.
\end{equation}
Since at least two of the three indices $i_1,i_2,i_3$ are equal, without loss of generality we treat here the case $i_2=i_3$. By \eqref{eq:semiorth} and the first equation of (\ref{key_assumption}), we have
\[
\mathfrak{C}\circ \mathfrak{p}_k = \mathfrak{p}_l' \circ \mathfrak{C}\circ \mathfrak{p}_k , \quad k\leq l,
\]
which further yields
\begin{align*}
    \widetilde{\mathfrak{p}}'_{i_2}\circ \mathfrak{C}\circ \widetilde{\mathfrak{p}}_{i_2} &= \Fp'_{i_2} \circ \cdots \circ \Fp'_{2d_f} \circ\FC \circ  \Fp_{i_2} \circ \cdots \circ \Fp_{2d_f}\\
    &= \FC \circ  \Fp_{i_2} \circ \cdots \circ \Fp_{2d_f}\\
    &= \FC \circ \widetilde{\Fp}_{i_2}.
\end{align*}
Consequently, (\ref{each_term}) follows:
\begin{align*}
    \widetilde{\mathfrak{p}}_{i_1}\circ \mathfrak{C}^{-1}\circ \widetilde{\mathfrak{p}}'_{i_2}\circ \mathfrak{C}\circ \widetilde{\mathfrak{p}}_{i_2} & =  \widetilde{\mathfrak{p}}_{i_1}\circ \mathfrak{C}^{-1}\circ \mathfrak{C}\circ \widetilde{\mathfrak{p}}_{i_2}\\
    & = \widetilde{\mathfrak{p}}_{i_1}\circ \widetilde{\mathfrak{p}}_{i_2}\\
    & = \widetilde{\Fp}_{\mathrm{min}\{i_1,i_2\}}.
\end{align*}
Finally, (\ref{8terms}) is deduced immediately from (\ref{each_term}), which completes the proof of (b).
\end{proof}

In many cases, it is more natural to work with a motivic perverse filtration, rather than a motivic decomposition. This is because they carry the same information, but the extra complexity of obtaining the orthogonal projectors from the semi-orthogonal projectors, as in the proof of Proposition \ref{prop2.1}, is formal and not essential. This in particular is the case for compactified Jacobian fibrations \cite[Section 2.5.5]{MSY} as we will also review below, where the semi-orthogonal projectors are of better form; see Theorem \ref{thm2.2}.

\subsection{Compactified Jacobians}\label{sec2.2}

Throughout this section, we work with the general setting that $C \to B$ is a flat family of integral projective curves of arithmetic genus $g$ with planar singularities over an irreducible base $B$; we will specialize to the case of spectral curves in Section \ref{Sec2_last}. We assume that the total space $C$ is nonsingular with a multi-section 
\[
D \subset C \to B
\]
of degree $r$ which is finite and flat over $B$.
For convenience, we also assume that for any~$d$, the degree $d$ compactified Jacobian~$\overline{J}^d_C$ is a quasi-projective variety. This is, for example, satisfied when $C \to B$ arises as a linear system of a nonsingular surface, where the quasi-projectivity follows from the fact that the compactified Jacobian arises as the moduli space of certain semistable torsion sheaves on the surface; see \emph{e.g.}~\cite{Lp}. We further assume that~$\overline{J}^d_C$ is nonsingular. In order to work with the universal family, we also need to consider the moduli stack of the degree $d$ compactified Jacobian, which is a $\BG_m$-gerbe over $\overline{J}^d_C$. Following \cite[Section 4.2]{MSY}, we rigidify this stack by imposing the condition that the universal family is \emph{trivialized} along the multi-section; this gives a Deligne--Mumford stack which is a $\mu_r$-gerbe over the scheme:
\[
\overline{\CJ}^d_C \to \overline{J}_C^d;
\]
see \cite[Proposition 4.1]{MSY}. The rigidification condition we imposed implies that there is a universal sheaf $\CF^d$ over $C\times_B \overline{\CJ}^d_C$ satisfying
\[
\mathrm{det}\left( p_{J*}\left(\CF^d|_{\overline{\CJ}^d_C\times_C D} \right)  \right) \simeq \CO_{\overline{\CJ}^d_C} \in \mathrm{Pic}\left( \overline{\CJ}^d_C \right).
\]
Here $p_J$ is the projection to the compactified Jacobian factor.

Now we consider two integers $d,e$. Arinkin's work \cite{A2} shows that the universal sheaves~$\CF^d, \CF^e$ induce a pair of Fourier--Mukai kernels
\begin{equation}\label{Poincare}
\CP_{e,d} \in D^\mathrm{b}\mathrm{Coh}\left( \overline{\CJ}^e_C\times \overline{\CJ}^d_C\right), \quad \CP^{-1}_{e,d} \in D^\mathrm{b}\mathrm{Coh}\left( \overline{\CJ}^d_C\times \overline{\CJ}^e_C\right).
\end{equation}
The associated Fourier--Mukai transforms are equivalences of categories
\[
\mathrm{FM}_{\CP_{e,d}}: D^\mathrm{b}\mathrm{Coh}(\overline{\CJ}^e_C)_{(-d)} \xrightarrow{~~\simeq~~} D^\mathrm{b}\mathrm{Coh}(\overline{\CJ}^{d}_C)_{(e)},\quad \mathrm{FM}_{\CP^{-1}_{e,d}}=\mathrm{FM}^{-1}_{\CP_{e,d}},
\]
where $D^\mathrm{b}\mathrm{Coh}(\overline{\CJ}^m_C)_{(k)}$ is the isotypic category consisting of objects for which the action of $\mu_r$ on the fibers is given by the character $\lambda \mapsto \lambda^k$ of $\mu_r$. The objects (\ref{Poincare}) induce (Chow-theoretic) Fourier transforms
\begin{align*}
    \mathfrak{F}_{e,d} & := \td\left(-T_{\overline{\CJ}_C^e \times_B \overline{\CJ}^d_C}\right) \cap \tau(\CP_{e,d}) \in \Chow_*(\overline{\CJ}_C^e \times_B \overline{\CJ}^d_C) = \Chow_*(\overline{J}_C^e \times_B \overline{J}^d_C),  \\
    \mathfrak{F}_{e,d}^{-1} & := \td(-(\pi_d \times_B \pi_e)^*T_B) \cap \tau(\CP_{e,d}^{-1}) \in \Chow_*(\overline{\CJ}_C^d \times_B \overline{\CJ}^e_C) = \Chow_*(\overline{J}_C^d \times_B \overline{J}^e_C).
\end{align*}
Here $\pi_\bullet$ is the natural projection to the base $B$, $T_{\overline{\CJ}_C^e \times_B \overline{\CJ}^d_C}$ is the virtual tangent bundle, $\tau(-)$ is the tau-class used in the Riemann--Roch theorem of Baum--Fulton--MacPherson \cite{BFM}, and we identify the (rational) Chow groups of the coarse moduli spaces and the gerbes; we refer to \cite[Section 4]{MSY} for more details. 

The correspondences $\mathfrak{F}_{e,d}, \mathfrak{F}^{-1}_{e,d}$ are of mixed degrees, and we have decompositions
\begin{gather*}
\mathfrak{F}_{e,d} = \sum_{i \geq 0} (\mathfrak{F}_{e,d})_i, \quad (\FF_{e, d})_i \in \Corr_B^{i - g}(\overline{J}_C^e, \overline{J}_C^d), \\
\mathfrak{F}_{e,d}^{-1} = \sum_{i \geq 0} (\mathfrak{F}^{-1}_{e,d})_i, \quad (\FF^{-1}_{e, d})_i \in \Corr_B^{i - g}(\overline{J}_C^d, \overline{J}_C^e).
\end{gather*}
For a fixed $d$, in order to define a motivic perverse filtration associated with $\pi_d: \overline{J}^d_C \to B$ using the Fourier transform, we need to choose a degree $e$, whose associated compactified Jacobian fibration will serve as the ``dual" side; see Step 1 of \cite[Section 4.4]{MSY}. In this paper, we choose this degree uniformly to be $1$; this choice will be important for applying the Abel--Jacobi map $C \to \overline{J}^1_C$ to build a connection between the Fourier transform and the normalized Chern classes of a universal bundle for a Hitchin-type moduli space. We define
\begin{equation}\label{motivic_p_k}
\mathfrak{p}_k:= \sum_{i \leq k} (\FF_{1,d})_i \circ (\FF^{-1}_{1,d})_{2g - i} \in  \Corr_B^0(\overline{J}_C^d , \overline{J}^d_C).
\end{equation}

The following theorem proven in \cite{MSY} confirmed the motivic decomposition conjecture for the compactified Jacobian fibration associated with $C\to B$, with the projectors constructed above.

\begin{thm}[{\cite[Corollary 4.6(i)]{MSY}}]\label{thm2.2}
For every $0\leq k \leq 2g$, the self-correspondence $\mathfrak{p}_k$ is a projector for $\pi_d: \overline{J}^d_C \to B$, and the motives
\[
P_kh(\overline{J}^d_C):= (\overline{J}^d_C, \mathfrak{p}_k, 0) \in \mathrm{CHM}(B), \quad 0\leq k \leq 2g
\]
form a motivic perverse filtration in the sense of Definition \ref{def_motivic_P}. Moreover, every submotive~$P_{k}h(\overline{J}^d_C)$ admits a canonical orthogonal complement $Q_{k+1}h(\overline{J}^d_C)$:
\[
h(\overline{J}^d_C) = P_kh(\overline{J}^d_C) \oplus Q_{k+1}h(\overline{J}^d_C), \quad Q_{k+1}h(\overline{J}^d_C) = (\overline{J}^d_C, \mathfrak{q}_{k+1}, 0)
\]
with
\[
\mathfrak{q}_{k+1} := [\Delta_{\overline{J}^d_C/B}] - \mathfrak{p}_k = \sum_{i\geq k+1}
(\FF_{1,d})_i \circ (\FF^{-1}_{1,d})_{2g - i}.
\]
\end{thm}

The orthogonal complements $Q_\bullet h(\overline{J}^d_C)$ allow us to directly formulate the condition that a correspondence to $h(\overline{J}^d_C)$ factors through a certain piece of the motivic perverse filtration; we refer to the statements in Theorem \ref{thm2.3} below as examples.

Now we consider two natural operations on the relative Chow motive $h(\overline{J}^d_C)$. The first is induced by the correspondence $(\mathfrak{F}_{1,d})_k$:
\begin{equation} \label{eq:fourierk}
(\mathfrak{F}_{1,d})_k: h(\overline{J}^1_C)(g-k) \to h(\overline{J}^d_C).
\end{equation}
The second is the multiplicative structure given by cup-product:
\[
\cup: h(\overline{J}^d_C) \times h(\overline{J}^d_C) \to h(\overline{J}^d_C).
\]

The following theorem shows the compatibility between the motivic perverse filtration of Theorem \ref{thm2.2} and the two operations above.

\begin{thm}\label{thm2.3}
The following hold for $\pi_d: \overline{J}^d_C \to B$.
\begin{enumerate}
    \item[(a)] The morphism $(\mathfrak{F}_{1,d})_k$ in \eqref{eq:fourierk} factors through
\[
(\mathfrak{F}_{1,d})_k: h(\overline{J}^1_C)(g-k) \to P_kh(\overline{J}^d_C).
\]
In other words, the following composition is zero:
\[
h(\overline{J}^1_C)(g-k) \xrightarrow{(\mathfrak{F}_{1,d})_k} h(\overline{J}^d_C) \to Q_{k+1}h(\overline{J}^d_C).
\]
\item[(b)] The motivic perverse filtration is multiplicative:
\[
\cup: P_k h(\overline{J}^d_C) \times P_lh(\overline{J}^d_C) \to P_{k+l}h(\overline{J}_C^d).
\]
In other words, the following composition is zero:
\[
P_k h(\overline{J}^d_C) \times P_lh(\overline{J}^d_C) \xrightarrow{\cup} h(\overline{J}^d_C) \to Q_{k+l+1}h(\overline{J}_C^d).
\]
\end{enumerate}
\end{thm}

Statement (a) was proven in \cite[Corollary 4.6(iv)]{MSY} by setting $e=1$. However, (b) does not follow directly from the multiplicativity result \cite[Corollary 4.6(iii)]{MSY} (since we cannot make $e_1 = e_2=e_1+e_2 = 1$); we postpone the proof to Section \ref{Section2.2}. The following result is an immediate consequence of the multiplicativity.

\begin{cor}\label{cor2.4}
Assume that $\alpha \in \mathrm{CH}^l(\overline{J}^d_C)$ lies in the $k$-th piece of the perverse filtration, \emph{i.e.},
\[
\alpha \in \mathrm{CH}^l(P_kh(\overline{J}^d_C)).
\]
Then cupping with $\alpha$ interacts with the motivic perverse filtration as follows:
\[
\alpha: P_ih(\overline{J}^d_C) \to P_{i+k}h(\overline{J}^d_C)(l).
\]
\end{cor}

\begin{proof}
It suffices to prove the relation 
\[
\mathfrak{q}_{i + k +1}\circ \Delta_*\alpha \circ \mathfrak{p}_i = 0
\]
on the relative product $\overline{J}^d_C\times_B \overline{J}^d_C$, where $\Delta: \overline{J}^d_C \hookrightarrow \overline{J}^d_C \times_B \overline{J}^d_C$ is the diagonal embedding. Theorem \ref{thm2.3}(b) reads
\[
\mathfrak{q}_{i + k +1}\circ [\Delta^{\mathrm{sm}}_{\overline{J}_C^d/B}]\circ \left( \mathfrak{p}_k \times \mathfrak{p}_i \right)=0
\]
which is a relation on the relative triple product. Therefore, we have
\begin{align*}
    0 &= p_{23*}\left(p_1^* \alpha \cap \left( \mathfrak{q}_{i + k +1}\circ [\Delta^{\mathrm{sm}}_{\overline{J}_C^d/B}] \circ \left( \mathfrak{p}_k \times \mathfrak{p}_i \right)\right) \right)\\
    & =\mathfrak{q}_{i + k +1}\circ p_{23*}\left(p_1^* \alpha \cap \left( [\Delta^{\mathrm{sm}}_{\overline{J}_C^d/B}] \circ \left( \mathfrak{p}_k \times \mathfrak{p}_i \right)\right) \right)\\
    & = \mathfrak{q}_{i + k +1}\circ \Delta_*{\mathfrak{p}_k(\alpha)} \circ \mathfrak{p}_i\\
    & = \mathfrak{q}_{i + k +1}\circ \Delta_*{\alpha} \circ \mathfrak{p}_i.
\end{align*}
Here $p_{ij},p_i$ are the natural projections from the relative triple product $\overline{J}^d_C\times_B \overline{J}^d_C\times_B \overline{J}^d_C$ to the corresponding factors. This completes the proof.
\end{proof}

Next, we discuss the motivic $\chi$-independence phenomenon for compactified Jacobian fibrations. For two integers $d,d'$, we consider the correspondence
\[
\mathfrak{C}_{d,d'}:=  \sum_{i=0}^{2g} (\FF_{1,d'})_i \circ (\FF^{-1}_{1,d})_{2g - i} \in  \Corr_B^0(\overline{J}_C^d, \overline{J}^{d'}_C).
\]

Recall that $r$ is the degree of the multisection $D\subset C$ we fixed at the beginning.

\begin{thm}\label{thm2.5}
Assume $d,d'$ are two integers coprime to $r$. Then $\mathfrak{C}_{d,d'}$ induces an isomorphism 
\[
\mathfrak{C}_{d,d'}: h(\overline{J}^d_C) \xrightarrow{~~\simeq~~} h(\overline{J}^{d'}_C)
\]
preserving the motivic perverse filtrations
\[
\mathfrak{C}_{d,d'}: P_kh(\overline{J}^d_C) \xrightarrow{~~\simeq~~} P_kh(\overline{J}^{d'}_C).
\]
Moreover, the inverse of $\mathfrak{C}_{d,d'}$ is given by $\mathfrak{C}_{d',d}$,
\begin{equation} \label{eq:chiisom}
\mathfrak{C}_{d',d} \circ \mathfrak{C}_{d,d'} = [\Delta_{\overline{J}^d_C/B}], \quad \mathfrak{C}_{d,d'} \circ \mathfrak{C}_{d',d} = [\Delta_{\overline{J}^{d'}_C/B}].
\end{equation}
\end{thm}

Both Theorem \ref{thm2.3}(b) and Theorem \ref{thm2.5} are proven in the next section, which rely on a generalized version of the Fourier vanishing established in \cite{MSY}.

\subsection{Proofs of Theorem \ref{thm2.3}(b) and Theorem \ref{thm2.5}}\label{Section2.2} In \cite[Section 4.4]{MSY}, we establish a series of relations between the Fourier components, known as Fourier vanishing (FV). For two integers $d, e$, this vanishing states that
\[
\tag{FV} (\FF_{e, d}^{-1})_i \circ (\FF_{e, d})_j = 0, \quad i + j < 2g.
\]
In this section, we prove a generalized Fourier vanishing given as follows: 

\begin{prop}
Let $d, d', e, e'$ be integers. Then if $e$ is coprime to $r$, we have
\begin{equation*}
\tag{FV1} (\FF_{e', d}^{-1})_i \circ (\FF_{e, d})_j = 0, \quad i + j < 2g.
\end{equation*}
If $d$ is coprime to $r$, we have
\begin{equation*}
\tag{FV2} (\FF_{e, d})_i \circ (\FF_{e, d'}^{-1})_j = 0, \quad i + j < 2g.
\end{equation*}
\end{prop}

\begin{proof}
The proof of (FV1) follows closely that of (FV) in \cite[Proof of Corollary 4.6(i)]{MSY}, where we used the Adams operations $\psi^N(-)$ of \cite[Section 4]{AGP} to scale the Poincar\'e sheaf $\CP_{e, d}$ in $K$-theory. By applying the tau-class $\tau(-)$ to the $K$-theory classes supported in codimension~$g$:
\[
\CP^{-1}_{e, d} \circ \psi^N(\CP_{e,d}) \in K_*(\overline{\mathcal{J}}_C^e \times_B \overline{\mathcal{J}}_C^e)
\]
for infinitely many $N$, we were able to obtain (FV) for each pair of $i, j$ satisfying $i + j < 2g$. The crucial point here is that the functor $\tau(-)$ of \cite{EG2} is only covariant under proper \emph{representable} morphisms of quotient stacks. By choosing
\[
N \equiv 1 \pmod{r}
\]
we made sure that the class
\[
p_{12}^*\psi^N(\CP_{e, d}) \otimes p_{23}^*\CP_{e, d}^{-1} \in K_*(\overline{\mathcal{J}}_C^e \times_B \overline{\mathcal{J}}_C^d \times_B \overline{\mathcal{J}}_C^e)_{(d, 0, -d)}
\]
descends to a class in $K_*(\overline{\mathcal{J}}_C^e \times_B \overline{J}_C^d \times_B \overline{\mathcal{J}}_C^e)$ in order to apply $\tau(-)$ to the proper representable morphism
\[
p_{13}: \overline{\mathcal{J}}_C^e \times_B \overline{J}_C^d \times_B \overline{\mathcal{J}}_C^e \to \overline{\mathcal{J}}_C^e \times_B \overline{\mathcal{J}}_C^e.
\]
Here $p_{ij}$ are the natural projections to the corresponding factors.

Now, under the assumptions of (FV1), we consider the $K$-theory classes supported in codimension $g$:
\[
\CP^{-1}_{e', d} \circ \psi^N(\CP_{e,d}) \in K_*(\overline{\mathcal{J}}_C^e \times_B \overline{\mathcal{J}}_C^{e'}).
\]
Since $(e, r) = 1$, we have $ae \equiv 1 \pmod{r}$ for some integer $a$. Then by taking
\[
N = aN', \quad N' \equiv e' \pmod{r},
\]
we get an infinite sequence of classes
\[
p_{12}^*\psi^N(\CP_{e, d}) \otimes p_{23}^*\CP_{e', d}^{-1} \in K_*(\overline{\mathcal{J}}_C^e \times_B \overline{\mathcal{J}}_C^d \times_B \overline{\mathcal{J}}_C^{e'})_{(ae'd, 0, -d)}
\]
which descend to $K_*(\overline{\mathcal{J}}_C^e \times_B \overline{J}_C^d \times_B \overline{\mathcal{J}}_C^{e'})$. The argument of \cite[Proof of Corollary 4.6(i)]{MSY} then applies \emph{verbatim} to yield (FV1).

For the proof of (FV2), we can either repeat the whole argument by switching the composition order of the Fourier transform and its inverse, or simply observe that the transpose~of
\[
(\FF_{e, d})_i \circ (\FF_{e, d'}^{-1})_j
\]
is precisely
\begin{equation} \label{eq:transpose}
(\FF^{-1}_{d', e})_j \circ (\FF_{d, e})_i.
\end{equation}
This is an immediate consequence of the symmetry in Arinkin's formula \cite[(1.1)]{A2} for the Poincar\'e sheaf; see also \cite[Footnote 12]{MSY}. Therefore, (FV2) follows from (FV1) and \eqref{eq:transpose} by replacing~$e$ with $d$, $e'$ with $d'$, $d$ with $e$, and by switching $i$ and $j$.
\end{proof}

We are ready to prove Theorems \ref{thm2.3}(b) and \ref{thm2.5}. We begin with the latter.

\begin{proof}[Proof of Theorem \ref{thm2.5}]
We first prove the identities \eqref{eq:chiisom}. We expand
\begin{align}
[\Delta_{\overline{J}^d_C/B}] & = \FF_{1, d} \circ \FF^{-1}_{1, d'} \circ \FF_{1, d'} \circ \FF^{-1}_{1, d} \nonumber\\
& = \sum_{i, j, l, m} (\FF_{1, d})_i \circ (\FF^{-1}_{1, d'})_j \circ (\FF_{1, d'})_l \circ (\FF^{-1}_{1, d})_m. \label{eq:ijlm}
\end{align}
For degree reasons the terms in \eqref{eq:ijlm} vanish unless $i + j+ l + m = 4g$. Applying the Fourier vanishing (FV2) with $d, d'$ coprime to $r$, we find that the nonvanishing terms in \eqref{eq:ijlm} must also satisfy $i + j \geq 2g$, $l + m \geq 2g$. Hence both inequalities are actual equalities and we have
\[
[\Delta_{\overline{J}^d_C/B}] = \left(\sum_{i = 0}^{2g}(\FF_{1, d})_i \circ (\FF^{-1}_{1, d'})_{2g - i}\right) \circ \left(\sum_{l = 0}^{2g}(\FF_{1, d'})_l \circ (\FF^{-1}_{1, d})_{2g - l}\right) = \FC_{d', d} \circ \FC_{d, d'},
\]
which proves the first identity in \eqref{eq:chiisom}. The second identity follows by symmetry.

It remains to verify that $\FC_{d, d'}, \FC_{d',d}$ preserve the motivic perverse filtrations. This amounts to showing the relations (\ref{key_assumption}) where $\Fp_k, \Fq_{k + 1}$ (resp.~$\Fp'_k, \Fq'_{k + 1}$) denote the projectors for $\overline{J}^d_C$ (resp.~$\overline{J}^{d'}_C$). By symmetry it suffices to prove the first identity:
\[
\Fq'_{k + 1} \circ\mathfrak{C}_{d,d'}\circ\Fp_k=0.
\]
The left-hand side can be expressed as
\begin{equation} \label{expand}
 \left(\sum_{i \geq k + 1}(\FF_{1, d'})_i \circ (\FF^{-1}_{1, d'})_{2g - i}\right) \circ \left(\sum_{j = 0}^{2g}(\FF_{1, d'})_j \circ (\FF^{-1}_{1, d})_{2g - j}\right) \circ \left(\sum_{m \leq k}(\FF_{1, d})_m \circ (\FF^{-1}_{1, d})_{2g - m}\right).
\end{equation}
This time applying (FV1) with $e = 1$ to (\ref{expand}), we see that the nonvanishing terms must satisfy
\[
2g - i + j \geq 2g, \quad 2g - j + m \geq 2g, \quad i \geq k + 1, \quad m \leq k,
\]
or, equivalently,
\[
k \geq m \geq j \geq i \geq k + 1.
\]
This is absurd, and hence \eqref{expand} vanishes.
\end{proof}

The proof of Theorem \ref{thm2.3}(b) is slightly more complicated but follows a similar pattern. For two integers $e_1, e_2$, recall from \cite[Proof of Corollary 4.6(iii)]{MSY} the Chow-theoretic convolution kernel (with a different notation to avoid confusion with the correspondence $\FC_{d, d'}$)
\[
\FK_{e_1, e_2, d} \in \Corr_B^{\geq -g}(\overline{J}_C^{e_1}, \overline{J}_C^{e_2}; \overline{J}_C^{e_1 + e_2}),
\]
which satisfies
\begin{equation} \label{eq:convker}
\FF_{e_1 + e_2, d} \circ \FK_{e_1, e_2, d} \circ \left(\FF^{-1}_{e_1, d} \times \FF^{-1}_{e_2, d}\right) = [\Delta^{\mathrm{sm}}_{\overline{J}^d_C/B}].
\end{equation}

\begin{proof}[Proof of Theorem \ref{thm2.3}(b)]
We need to show the vanishing
\begin{equation} \label{eq:multtwist}
\Fq_{k + l + 1} \circ [\Delta^{\mathrm{sm}}_{\overline{J}^d_C/B}] 
\circ (\Fp_k \times \Fp_l) = 0 \in \Corr_B^0(\overline{J}^d_C, \overline{J}^d_C; \overline{J}^d_C).
\end{equation}
By \eqref{eq:convker} we have
\begin{align}
\Fq_{k + l + 1} \circ [\Delta^{\mathrm{sm}}_{\overline{J}^d_C/B}] 
\circ (\Fp_k \times \Fp_l) & = \Fq_{k + l + 1} \circ \FF_{e_1 + e_2, d} \circ \FK_{e_1, e_2, d} \circ \left((\FF^{-1}_{e_1, d} \circ \Fp_k) \times (\FF^{-1}_{e_2, d} \circ \Fp_l)\right) \nonumber\\
& = \left(\sum_{i_3 \geq k + l + 1} (\FF_{1, d})_{i_3} \circ (\FF^{-1}_{1, d})_{2g - i_3}\right) \circ \left(\sum_{j_3}(\FF_{e_1 + e_2, d})_{j_3}\right) \circ \FK_{e_1, e_2, d} \label{eq:cupexpand}\\
& \quad\quad \circ \left(\left(\left(\sum_{j_1}(\FF^{-1}_{e_1, d})_{j_1}\right) \circ \left(\sum_{i_1 \leq k}(\FF_{1, d})_{i_1} \circ (\FF^{-1}_{1, d})_{2g - i_1}\right)\right)\right. \nonumber\\
& \quad\quad\quad \times \left.\left(\left(\sum_{j_2}(\FF^{-1}_{e_2, d})_{j_2}\right) \circ \left(\sum_{i_2 \leq l}(\FF_{1, d})_{i_2} \circ (\FF^{-1}_{1, d})_{2g - i_2}\right)\right)\right). \nonumber
\end{align}
So far the choices of $e_1, e_2$ have been arbitrary. But if we choose $e_1, e_2$ such that $e_1 + e_2$ is coprime to $r$, then by applying (FV1) with $e = 1$ or $e = e_1 + e_2$ we see that the nonvanishing terms in \eqref{eq:cupexpand} must satisfy
\[
j_1 + i_1 \geq 2g, \quad j_2 + i_2 \geq 2g, \quad 2g - i_3 + j_3 \geq 2g.
\]
This implies
\[
j_1 \geq 2g - k, \quad j_2 \geq 2g - l, \quad j_3 \geq k + l + 1.
\]
Therefore we have
\begin{multline} \label{eq:degcompare}
\Fq_{k + l + 1} \circ [\Delta^{\mathrm{sm}}_{\overline{J}^d_C/B}] 
\circ (\Fp_k \times \Fp_l) = \Fq_{k + l + 1} \circ \left(\sum_{j_3 \geq k + l + 1}(\FF_{e_1 + e_2, d})_{j_3}\right) \circ \FK_{e_1, e_2, d} \\
\circ \left(\left(\left(\sum_{j_1 \geq 2g - k}(\FF^{-1}_{e_1, d})_{j_1}\right) \circ \Fp_k\right) \times \left(\left(\sum_{j_2 \geq 2g - l}(\FF^{-1}_{e_2, d})_{j_2}\right) \circ \Fp_l\right)\right)
\end{multline}
Comparing degrees, we find that the right-hand side of \eqref{eq:degcompare} lies in
\[
\Corr_{B}^{\geq -g + (k + l + 1 - g) + (g - k) + (g - l)}(\overline{J}^d_C, \overline{J}^d_C; \overline{J}^d_C) = \Corr_B^{\geq 1}(\overline{J}^d_C, \overline{J}^d_C; \overline{J}^d_C),
\]
while the left-hand side lies in $\Corr^0_B(\overline{J}^d_C, \overline{J}^d_C; \overline{J}^d_C)$. This shows the desired vanishing \eqref{eq:multtwist}.
\end{proof}

\subsection{Parabolic Higgs bundles: general residue}\label{Sec2_last}

Theorems \ref{thm2.2}, \ref{thm2.3}, and \ref{thm2.5} in Section~\ref{sec2.2} are analogues of our main theorems in the general setting of compactified Jacobians. For our purpose, we specialize to the case where $C \to B$ arises as a family of spectral curves. The main results of this section are Theorems \ref{thm2.7}, \ref{thm2.8}, and \ref{thm2.9} which treat the case of the Hitchin system associated with parabolic Higgs bundles with general residue.

\subsubsection{Normalized Chern characters}

For the curve $\Sigma$, we consider the nonsingular surface
\[
S:= \mathrm{Tot}_\Sigma(\omega_\Sigma(p)).
\]
Now we work with a family of curves $C \to B$ as in Section \ref{sec2.2}, imposing the extra condition that they lie in the linear system $|n\Sigma|$ of the surface $S$; in particular, every curve $C_b \subset S$ is a degree $n$ finite cover over the zero-section $\Sigma$. We have a natural evaluation map
\[
\mathrm{ev}: C \to S \times B.
\]
The composition of $\mathrm{ev}$ with the projection $S\times B \to \Sigma \times B$ gives the spectral cover together with a (degree $n$) multisection:
\[
\pi: C \to \Sigma \times B, \quad  D:= \pi^{-1}(p \times B).
\]
For each curve $C_b$ with a degree $n$ spectral cover $\pi_b: C_b \to \Sigma$, the multisection over this fiber is ${D}_b = \pi_b^{-1}(p)$.

Next, we introduce the normalized Chern character of Section \ref{Sect0.2}. Let $M^{\mathrm{mero}}$ be the moduli space of stable meromorphic Higgs bundles on the curve $\Sigma$, which parameterizes slope stable (meromorphic) Higgs bundles:
\[
(\CE, \theta), \quad \theta: \CE \to \CE\otimes \omega_\Sigma(p), \quad \mathrm{rk}(\CE) = n, \quad \mathrm{deg}(\CE) = d.
\]
Here we call them meromorphic Higgs bundles since they can be viewed as Higgs bundles which are allowed to have a simple pole at $p$. The variety $M^{\mathrm{mero}}$ is nonsingular, and the coprime condition $(n,d)=1$ ensures the existence of a universal bundle $\CU$ on $\Sigma \times M^{\mathrm{mero}}$. The reason we work first with this moduli space is that all the moduli spaces in Section~\ref{Sec1.2}, including the original moduli space $M_{n,d}$, admit natural morphisms to $M^{\mathrm{mero}}$, and the pullback of the universal bundle gives a universal bundle on each moduli space. So the normalized Chern~character
\[
\widetilde{\mathrm{ch}}(\CU) \in \mathrm{CH}^* ( \Sigma \times M^{\mathrm{mero}} ),
\]
yields the normalized Chern character for each moduli space of Section \ref{Sec1.2}. This strategy was taken in \cite[Section 5.4]{MSY}.

In cohomology, the normalization is given by a class
\begin{equation*}
\delta := \pi^*_\Sigma\delta_\Sigma + \pi^*_M\delta_M \in   H^2(\Sigma\times M^{\mathrm{mero}}, \BQ),
\end{equation*}
with $\pi_\bullet$ the natural projections, satisfying that 
\[
c_1(\CU) + \delta \in H^1(\Sigma, \BQ) \otimes H^1(M^{\mathrm{mero}}, \BQ) \subset H^2(\Sigma\times M^{\mathrm{mero}}, \BQ);
\]
this determines $\delta$ completely. Then the normalized Chern character is defined to be 
\[
\widetilde{\mathrm{ch}}(\CU): = \mathrm{ch}(\CU)\cup \mathrm{exp}\left(\delta \right) \in H^*(\Sigma \times M^{\mathrm{mero}}, \BQ),
\]
which is canonical and independent of the choice of $\CU$. We refer to \cite[Section 0.3]{dCMS} for more details.

For our purpose, we need to lift $\delta$ (associated with a universal bundle $\CU$) canonically to a divisor
\begin{equation}\label{sigma}
\delta: = \pi_\Sigma^* D_\Sigma + \pi_M^* D_M \in \mathrm{CH}^1(\Sigma \times M^{\mathrm{mero}}).
\end{equation}
We first determine $D_M$ using the marked point $p\in \Sigma$:
\[
D_M: = c_1(\CU)|_{p\times M^{\mathrm{mero}}} \in \mathrm{CH}^1(M^{\mathrm{mero}}).
\]
Now we determine $D_\Sigma$, which is eventually written as a linear combination of the point class of~$p$ and the canonical divisor $K_{\Sigma}$ on the curve $\Sigma$. Pick a nonsingular spectral curve \mbox{$C_b \subset \mathrm{Tot}_\Sigma(\omega_{\Sigma}(p))$}; a Higgs bundle in $M^{\mathrm{mero}}$ associated with the spectral curve $C_b$ can be represented by a line bundle $\CL_{C_b}$ on $C_b$ of degree $l(n,d)$, where $l(n,d)$ is given by a Riemann--Roch calculation 
\[
l(n,d)= \frac{n(n-1)(2g-1)}{2} +d.
\]
In particular, we have
\[
c_1(\CL_{C_b}) - \frac{l(n,d)}{n} \cdot {D}_b =c_1(\CL_{C_b}) - \frac{l(n,d)}{n} \cdot \pi_b^* [p]= 0 \in H^2(C_b, \BQ).
\]
Finally, there is a unique constant $\lambda_0 \in \frac{1}{2n}\BZ$ such that
\[
c_1(\CU) -  \frac{l(n,d)}{n} \cdot \pi_b^* [p] + \lambda_0 \cdot K_\Sigma = 0 \in H^2(\Sigma, \BQ).
\]
We set
\[
D_\Sigma : = - \frac{l(n,d)}{n} \cdot  [p] + \lambda_0 \cdot K_\Sigma,
\]
in (\ref{sigma}), which defines the normalized Chern character
\begin{equation}\label{NCC}
\widetilde{\mathrm{ch}}(\CU): = \mathrm{ch}(\CU) \cup \mathrm{exp}(\delta) \in \mathrm{CH}^*(\Sigma \times M^{\mathrm{mero}}).
\end{equation}
We denote by $\widetilde{\mathrm{ch}}_k(\CU)$ its degree $k$ part:
\[
\widetilde{\mathrm{ch}}(\CU) = \sum_{k} \widetilde{\mathrm{ch}}_k(\CU), \quad\widetilde{\mathrm{ch}}_k(\CU) \in \mathrm{CH}^k(\Sigma \times M^{\mathrm{mero}}).
\]
The normalized Chern character (\ref{NCC}) does not depend on the choice of the universal bundle~$\CU$, and its image under the cycle class map recovers the normalized Chern character in cohomology. Indeed, two universal bundles $\CU, \CU'$ differ by a line bundle pulled back from~$M^{\mathrm{mero}}$. Therefore we have $\mathrm{ch}(\CU') = \mathrm{ch}(\CU)\cup \mathrm{exp}(\ell_M)$ with $\ell_M \in \mathrm{CH}^1(M^{\mathrm{mero}})$. By definition, $\CU,\CU'$ share the same normalized Chern character.

\subsubsection{Normalized Chern characters via Fourier transforms}

Next, we give an explicit description of the normalized Chern character for the moduli space~$M_\eta^{\mathrm{par}}$ for a general $\eta \in T$ as in Section \ref{Sec1.2}(a); it admits a compactified Jacobian fibration associated with the family of spectral curves $C_\eta \to W_\eta$. In other words, we have 
\[
M_\eta^{\mathrm{par}} \simeq \overline{J}^e_{C_\eta} \to W_\eta, \quad e=l(n,d).
\]
We consider the diagram
\begin{equation}\label{eqn20}
    \begin{tikzcd}
C_\eta \arrow[r, "\mathrm{AJ}"] \arrow[d, "\pi "]
&\overline{J}^1_{C_\eta} \\
\Sigma \times W_\eta
\end{tikzcd}
\end{equation}
where the horizontal map is the Abel--Jacobi map
\[
\mathrm{AJ}: C_\eta \to \overline{J}^1_{C_\eta},\quad \quad  ( x \in C_b ) \mapsto \mathfrak{m}^\vee_{x/C_b}.
\]
The diagram (\ref{eqn20}) yields a relative correspondence over $W_\eta$:
\[
[C_\eta] \in \mathrm{CH}_{\dim C_\eta}( \Sigma \times \overline{J}^1_{C_\eta} ).
\]

\begin{prop}[\cite{MSY}]\label{Prop2.6}
Let $\CU$ be a universal bundle over $\Sigma \times M^{\mathrm{par}}_\eta \simeq \Sigma \times \overline{J}^e_{C_\eta}$. Then we~have
\[
\widetilde{\mathrm{ch}}(\CU) =  \mathfrak{F}_{1,e} \circ [C_\eta] \in \mathrm{CH}^*(\Sigma \times M^{\mathrm{par}}_\eta).
\]
\end{prop}

\begin{proof}
Applying the Grothendieck--Riemann--Roch (GRR) formula to \cite[Proposition 4.7]{MSY} as in the proof of \cite[Proposition 5.1(a)]{MSY}, we express the Chern character of the universal $1$-dimensional torsion free sheaf over the universal spectral curve $C_\eta$ in terms of $\mathfrak{F}_{1,e}$. Then the result follows from applying further the GRR formula along $\pi: C_\eta \to \Sigma\times W_\eta$, expressing the Chern character of a universal bundle over $\Sigma$ in terms of the Chern character of a universal torsion free sheaf over $C_\eta$ \cite[Section 5]{MSY}. 

Note that the term $ - \frac{l(n,d)}{n} \cdot  [p] $ in the normalization comes from the line bundle $\CN$ of \cite[Proposition 4.7]{MSY}, and the second term $\lambda_0\cdot K_\Sigma$ arise from the Todd contribution of the GRR calculation along $\pi: C_\eta\to \Sigma\times W_\eta$.
\end{proof}

\subsubsection{Main results for Higgs bundles with general residue}

We deduce the following three theorems for the moduli space $M^{\mathrm{par}}_\eta$, the first of which is an immediate consequence of Theorem~\ref{thm2.2}.

\begin{thm}\label{thm2.7}
The morphism $h_\eta: M^{\mathrm{par}}_\eta \to W_\eta$ admits a motivic perverse filtration in the sense of Definition \ref{def_motivic_P}, where the projectors are given in Theorem \ref{thm2.2}.
\end{thm}

The next theorem determines the perversity of the normalized Chern character.

\begin{thm}\label{thm2.8}
Taking cup-product with the normalized Chern character satisfies
\[
\widetilde{\mathrm{ch}}_k(\CU): P_ih(\Sigma \times M^{\mathrm{par}}_\eta) \to P_{i+k}h(\Sigma \times M^{\mathrm{par}}_\eta)(k) \in \mathrm{CHM}(\Sigma \times {W}_\eta).
\]
Here the motivic perverse filtration for $\Sigma \times M^{\mathrm{par}}_\eta \to \Sigma \times W_\eta$ is given by the pullback of the motivic perverse filtration for $h_\eta: M^{\mathrm{par}}_\eta \to W_\eta$.
\end{thm}

\begin{proof}
Recall the perverse filtration on $\mathrm{CH}_*(M^{\mathrm{par}}_\eta)$ induced by the motivic perverse filtration associated with the compactified Jacobian fibration $M^{\mathrm{par}}_\eta \to {W}_\eta$. We first note that
\[
\widetilde{\mathrm{ch}}_k(\CU)\in P_k\mathrm{CH}^{k}(\Sigma \times M^{\mathrm{par}}_\eta).
\]
This is a direct consequence of Proposition \ref{Prop2.6} combined with Theorem \ref{thm2.3}(a):
\[
\widetilde{\mathrm{ch}}_k(\CU) = (\mathfrak{F}_{1,e})_k\circ [C_\eta].
\]
Then applying Corollary \ref{cor2.4} to 
\[
\Sigma \times M^{\mathrm{par}}_\eta \to \Sigma \times {W}_\eta
\]
completes the proof of the theorem.
\end{proof}

The third theorem concerns the $\chi$-independence phenomenon and is immediately deduced from Theorem \ref{thm2.5}. 

\begin{thm}\label{thm2.9}
Assume $d,d'$ are two integers coprime to $n$ with $e = l(n, d), e' = l(n, d')$. Let $M^{\mathrm{par}}_\eta, {M^{\mathrm{par}}_\eta}'$ be the parabolic moduli spaces associated with the degrees $d,d'$ respectively. Then~$\mathfrak{C}_{e,e'}$ from Theorem \ref{thm2.5} induces an isomorphism 
\[
\mathfrak{C}_{e,e'}: h(M^{\mathrm{par}}_\eta) \xrightarrow{~~\simeq~~} h({M^{\mathrm{par}}_\eta}')
\]
preserving the motivic perverse filtrations
\[
\mathfrak{C}_{e,e'}: P_kh(M^{\mathrm{par}}_\eta) \xrightarrow{~~\simeq~~} P_kh({M^{\mathrm{par}}_\eta}').
\]
Moreover, the inverse of $\mathfrak{C}_{e,e'}$ is given by $\mathfrak{C}_{e',e}$.
\end{thm}

\subsection{Parabolic Higgs bundles: nilpotent residue} \label{sec2.5}

In the previous section, we have established the desired results for 
\[
h_\eta: M^{\mathrm{par}}_\eta \to W_\eta
\]
associated with parabolic Higgs bundles with general residue. The proofs in fact give the results for the family 
\[
h^\circ: M^{\mathrm{par}}_{T^\circ} \to W_{T^\circ}
\]
over $T^\circ$ introduced in Section \ref{Sec1.2}(a), which puts together the fibrations $h_\eta$ for $\eta \in T^\circ$. Since $\chi_p: M^{\mathrm{par}}_T \to T$ is a smooth morphism, by applying the specialization map for Chow groups via \eqref{sp_diagram}, all the results concerning relations of algebraic cycles hold for 
\[
h_0: M^{\mathrm{par}}_0 \to W_0.
\]

\begin{thm}\label{thm2.10}
The morphism $h_0: M^{\mathrm{par}}_0 \to W_0$ admits a motivic perverse filtration in the sense of Definition \ref{def_motivic_P}, where the projectors are induced by specializing the projectors of~$h_\eta: M^{\mathrm{par}}_\eta \to W_\eta$ with $\eta \in T^\circ$ given in Theorem \ref{thm2.7}.
\end{thm}

As we discussed at the beginning of this Section \ref{sec2.5}, the projectors we obtain from specialization automatically satisfy the conditions (a) and (c) in Definition \ref{def_motivic_P}. It remains to prove~(b), \emph{i.e.}, the homological realization recovers the perverse truncation functor. This will be proven in Section \ref{Sec4.1}.

\begin{thm}\label{thm2.11}
Let $\CU$ be a universal bundle over $\Sigma\times M^{\mathrm{par}}_0$. Then we have
\[
\widetilde{\mathrm{ch}}_k(\CU): P_ih(\Sigma \times M^{\mathrm{par}}_0) \to P_{i+k}h(\Sigma\times M^{\mathrm{par}}_0)(k) \in \mathrm{CHM}(\Sigma \times W_0).
\]
\end{thm}

\begin{proof}
Since the statement can be expressed in terms of a relation of algebraic cycles:
\[
\mathfrak{q}_{i+k+1}\circ \Delta_*\widetilde{\mathrm{ch}}_k(\CU) \circ \mathfrak{p}_i =0 
\]
with $\Delta: M^{\mathrm{par}}_0 \hookrightarrow M^{\mathrm{par}}_0\times_{W_0}M^{\mathrm{par}}_0$ the relative diagonal, it follows from specializing the corresponding relation for $h_\eta: M^{\mathrm{par}}_\eta \to W_\eta$ established in Theorem \ref{thm2.8}.
\end{proof}

Similarly, by a straightforward specialization argument, we obtain from Theorem \ref{thm2.9} the following theorem concerning the $\chi$-independence.

\begin{thm}\label{thm2.12}
Assume $d,d'$ are two integers coprime to $n$. Let $M^{\mathrm{par}}_0, {M^{\mathrm{par}}_0}'$ be the parabolic moduli spaces associated with the degrees $d,d'$ respectively. Then there is an isomorphism 
\[
 h(M^{\mathrm{par}}_0) \simeq h({M^{\mathrm{par}}_0}') 
\]
preserving the motivic perverse filtrations
\[
 P_kh(M^{\mathrm{par}}_0) \simeq P_kh({M^{\mathrm{par}}_0}').
 \]
\end{thm}

\begin{rmk}
We conclude this section by noting that the motivic perverse filtration we introduced in Theorem \ref{thm2.10} is independent of the choice of the $1$-dimensional family $T$. Indeed, we consider the Hitchin map associated with the total space of the moduli of stable parabolic Higgs bundles
\[
h: M^{\mathrm{par}} \to W
\]
in Section \ref{Sec1.2}. The Fourier transform can be defined over a Zariski open subset $W^\circ$ of the base $W$ formed by integral spectral curves, and the projectors $\mathfrak{p}_k$ we obtained in Theorem \ref{thm2.11} is the specialization of the projectors over $W^\circ$. Here we choose to work with a $1$-dimensional family $T$ for convenience (\emph{e.g.}~in the sheaf-theoretic argument of Section \ref{Section4}).
\end{rmk}

\section{Projectors and Springer theory}\label{Section3}
In this section, we study the correspondences given by (\ref{corr}) connecting $M^{\mathrm{par}}_0$ and $M_{n,d}$. Our main purpose is to prove Proposition \ref{prop3.4}, which yields a candidate for the motivic perverse filtration (\ref{final_perv}) for $f: M_{n,d} \to B$. For the approach, we reduce the Higgs case to a \emph{local model} where the correspondence (\ref{corr}) is understood via classical Springer theory. This lifts several cohomological results of \cite[Section 8]{HMMS} motivically. Springer-theoretic interpretations of parabolic Hitchin moduli spaces have previously been used in \cite{DGT} in the study of Kac polynomials.

\subsection{Springer theory}

For our purpose, we consider the case $G = \mathrm{PGL}_n$. We use $\mathfrak{g}, \mathfrak{b}, \mathfrak{t}$ to denote the Lie algebras of $G$, the Borel subgroup given by upper triangular matrices, and the maximal torus given by diagonal matrices, respectively. The Weyl group is the permutation group $\mathfrak{S}_n$ which acts on $\mathfrak{t}$ naturally. Let 
\[
\mathfrak{c}: = \mathfrak{g}\sslash G \simeq \mathfrak{t}\sslash \mathfrak{S}_n
\]
be the categorical quotient, and we have a natural morphism
\[
\chi: \mathfrak{g} \to \mathfrak{c}
\]
which can be viewed as calculating the characteristic polynomial of a matrix. We consider the Grothendieck--Springer resolution over the base $\mathfrak{c}$:
 \begin{equation*}
    \begin{tikzcd}
    \widetilde{\mathfrak{g}} \arrow[d, "q"] \arrow[r, ""] & \mathfrak{t} \arrow[d, ""] \\
    \mathfrak{g} \arrow[r, "\chi"] & \mathfrak{c}.
\end{tikzcd}
\end{equation*}
The fibers over $0\in \mathfrak{c}$ recovers the symplectic resolution of the nilpotent cone
\[
q_0: \widetilde{\CN}_0 \to \CN_0
\]
by the total cotangent bundle of the flag variety $\widetilde{\CN}_0 = T^*\mathrm{Fl}$, while the fibers over a general~$\zeta\in \mathfrak{c}$ (\emph{e.g.}~the corresponding matrix has distinct eigenvalues) is a natural quotient by the Weyl group:
\[
q_\zeta: \widetilde{\CN}_\zeta \to {\CN}_\zeta = \widetilde{\CN_\zeta}/\mathfrak{S}_n. 
\]
From this description, we see that any element $w\in \mathfrak{S}_n$ induces a self-correspondence of~$\widetilde{\CN}_\zeta$. As we will illustrate in the next section, the local analogue of the correspondence (\ref{corr}) is (the~$G$-equivariant version of) the following:
\begin{equation}\label{local_corr}
    \begin{tikzcd}
\mathrm{Fl} \arrow[r, "\iota"] \arrow[d, "\pi "]
& \widetilde{\CN}_{0} = T^*\mathrm{Fl} \\
\mathrm{pt}
\end{tikzcd}
\end{equation}
where $\iota$ is the closed embedding of the $0$-section. The following lemma collects a few standard results in Springer theory, which concerns the ($G$-equivariant) self-correspondences
\begin{equation}\label{corr_Fl}
[\mathrm{Fl}] \circ {^\mathfrak{t}}[\mathrm{Fl}] \in \mathrm{CH}_*(\widetilde{\CN_0} \times \widetilde{\CN_0}) = \mathrm{Corr}^*(\widetilde{\CN_0}, \widetilde{\CN_0})
\end{equation}
and
\begin{equation}\label{corr_Fl2}
    {^\mathfrak{t}}[\mathrm{Fl}] \circ [\mathrm{Fl}] \in  \mathrm{CH}_*(\mathrm{pt}) = \mathrm{Corr}^*(\mathrm{pt}, \mathrm{pt}).
    \end{equation}
See \cite[Section 3]{CG} and \cite[Section 1.5.10]{Yun_lecture} for references.

\begin{lem}\label{lem3.1}
The following statements hold for the $G$-equivariant correspondences $[\mathrm{Fl}], {^\mathfrak{t}}[\mathrm{Fl}]$.
\begin{enumerate}
    \item[(a)] The self-correspondence \eqref{corr_Fl} of $\widetilde{\CN}_0$ is the specialization of the self-correspondence of~$\widetilde{\CN}_\zeta$ induced by the longest element $w_0$ in the Weyl group $\mathfrak{S}_n$.
    \item[(b)] The self-correspondence of \eqref{corr_Fl2} of a point is 
    \[
    \varepsilon n!\cdot \mathrm{id} \in \mathrm{Corr}^0(\mathrm{pt}, \mathrm{pt}).
    \]
    Here $\varepsilon =\pm 1$ is a sign dependent on $n$.
    \item[(c)] There is a natural Weyl group $\mathfrak{S}_n$ action on the derived push-forward
\[
Rq_{0*} \BQ_{\widetilde{\CN_0}} \in D^\mathrm{b}_\mathrm{c}(\CN_0)
\]
with the anti-invariant part a shifted skyscraper sheaf
\[
\left(Rq_{0*} \BQ_{\widetilde{\CN_0}} \right)^{\mathrm{sgn}} =\BQ_0\left[ -2\binom{n}{2}\right].
\]
Moreover, the correspondences  
\[
[\mathrm{Fl}]: Rq_{0*} \BQ_{\widetilde{\CN}_0} \to \BQ_0\left[ -2\binom{n}{2} \right], \quad  {^\mathfrak{t}}[\mathrm{Fl}]: \BQ_0\left[ -2\binom{n}{2}\right]\to  Rq_{0*} \BQ_{\widetilde{\CN}_0} \in D^\mathrm{b}_\mathrm{c}(\CN_0)
\]
are projecting to the anti-invariant part and the natural inclusion from the anti-invariant part respectively. 
\end{enumerate}
\end{lem}

\begin{proof}
    Part (b) is the self-intersection of the Lagrangian zero-section 
    \[
    \iota: \mathrm{Fl} \hookrightarrow \widetilde{\CN}_0 = T^*\mathrm{Fl}
    \]
    where the constrant $n!$ is given by the Euler characteristic of the flag variety $\mathrm{Fl}$.

    To see (a) and (c), we recall the Grothendieck--Springer resolution
    \[
    q: \widetilde{\mathfrak{g}} \to \mathfrak{g}
    \]
    which is $G$-equivariant and small; the restriction of $q$ to the regular semisimple locus is a torsor under the Weyl group $\mathfrak{S}_n$. The derived pushforward
    \[
    Rq_* \mathbb{Q}_{\widetilde{\mathfrak{g}}} \in D^b_c(\mathfrak{g})
    \]
    is the intermediate extension of the local systems obtained from the regular semisimple locus, which naturally carries an $\mathfrak{S}_n$-action by the functoriality of intermediate extension. In particular, this induces the $\mathfrak{S}_n$-action on every Springer fiber. The isotypic component with respect to the longest element $w_0$ for the Springer fiber over $0 \in \CN_0 \subset \mathfrak{g}$ is given as in (c). For our purpose, it suffices to show that the generalized eigenspace with respect to $w_0$ is realized as the correspondence (\ref{corr_Fl}). As explained in \cite[Section 1.5.10]{Yun_lecture} and \cite[Remark 3.3.4]{Yun1}, the Springer action of $\mathfrak{S}_n$ can be realized by correspondences of the Steinberg variety $\widetilde{\mathfrak{g}} \times_{\mathfrak{g}} \widetilde{\mathfrak{g}}$. After restricting over $\CN_0$, the irreducible components of the Steinberg variety are parameterized by elements $w \in \mathfrak{S}_n$. More precisely, the Steinberg variety over $\CN_0$ admits a stratification into conormal bundles of $G$-orbits in $\mathrm{Fl}\times \mathrm{Fl}$ (see \cite[Corollary 3.3.5]{CG}), and the element $w_0$ corresponds to the largest open stratum whose conormal bundle has closure
    \[
    \mathrm{Fl} \times \mathrm{Fl} \subset \widetilde{\CN}_0 \times_{\CN_0} \widetilde{\CN}_0.
    \]
    Its induced correspondence is exactly (\ref{corr_Fl}), which proves the lemma.
\end{proof}



In the next section, we will discuss a global version of Lemma \ref{lem3.1} and use it to analyze the correspondence
\begin{equation}\label{MM}
[\widetilde{M}_0] \circ {^\mathfrak{t}}[\widetilde{M}_0] \in \mathrm{Corr}_{W_0}^*(M_0^{\mathrm{par}}, M_0^{\mathrm{par}}).
\end{equation}

\subsection{Hitchin moduli spaces}

Recall the moduli space $M^{\mathrm{mero}}$ of stable meromorphic Higgs bundles of Section \ref{Sec2_last}. It admits a natural evaluation map at the point $p\in \Sigma$ to the stack quotient $\mathfrak{g}/G$:
\begin{equation}\label{evaluation}
\mathrm{ev}_p: M^{\mathrm{mero}} \to \mathfrak{g}/G
\end{equation}
as we introduce below.

We denote by $M^{\mathrm{ab}}$ the \emph{abelian} Hitchin moduli space
\[
M^{\mathrm{ab}} = \mathrm{Pic}^0(\Sigma) \times H^0(\Sigma, \omega_\Sigma(p))
\]
which parameterize rank $1$ Higgs bundles. There is a natural action of $M^{\mathrm{ab}}$ on $M^{\mathrm{mero}}$ by
\[
(\CL, \sigma)\cdot (\CE, \theta) = (\CL\otimes \CE , \sigma+\theta), \quad (\CL, \sigma) \in M^{\mathrm{ab}}, \quad (\CE, \theta)\in M^{\mathrm{mero}}
\]
where the sum $\sigma +\theta$ is defined in the sense that we view both $\sigma, \theta$ as morphisms
\[
\CL\otimes \CE \to \CL\otimes \CE\otimes \omega_\Sigma(p).
\]
The quotient space
\[
\widehat{M}^{\mathrm{mero}}: = M^{\mathrm{mero}}/M^{\mathrm{ab}}
\]
is a nonsingular Deligne--Mumford stack, known as the moduli space of stable meromorphic~$\mathrm{PGL}_n$-Higgs bundles, which admits a natural evaluation map at the point $p$; see \cite[Section 2]{HT0} and \cite[Section 2.4]{dCHM1}. The evaluation map (\ref{evaluation}) for the original moduli space is the composition of the quotient map and the~$\mathrm{PGL}_n$-evaluation map:
\[
M^{\mathrm{mero}} \twoheadrightarrow \widehat{M}^{\mathrm{mero}} \to \mathfrak{g}/G.
\]

\begin{prop}\label{prop3.2}
    The morphism $\mathrm{ev}_p$ in \eqref{evaluation} is smooth and surjective.
\end{prop}

\begin{proof}
    The quotient map $M^{\mathrm{mero}} \twoheadrightarrow \widehat{M}^{\mathrm{mero}}$ is smooth. Therefore, the smoothness of the composition follows from the smoothness of the second evaluation map which has been proven in \cite[Proposition 4.1]{MS_HT} via deformation theory.
\end{proof}

\begin{rmk}
    We do not work directly with the evaluation map associated with $\mathrm{GL}_n$ for the following reasons. First, Lemma \ref{lem1.20} implies that the image of $\mathrm{ev}_p$ for a $\mathrm{GL}_n$-Higgs bundle has trace $0$; in particular the trace lies naturally in 
    \[
    \mathfrak{g}=\{\mathrm{tr} = 0\} \subset \mathfrak{gl}_n.
    \]
    Second, when we pass from the moduli stack of $\mathrm{GL}_n$-stable Higgs bundle to the corresponding moduli space, we need to rigidify the automorphism group $\BC^* \subset \mathrm{GL}_n$. Working with the evaluation map associated with $G = \mathrm{PGL}_n$ is more convenient which resolves both issues.
\end{rmk}

Now we may use the evaluation map to pull back natural morphisms in Springer theory. As a first example, the pullback of (\ref{evaluation}) via the ($G$-equivariant) Grothendieck--Springer resolution recovers the parabolic moduli space

\begin{equation*}
\begin{tikzcd}
M^{\mathrm{par}} \arrow[r, "\mathrm{ev}_p"] \arrow[d, ""]
& \widetilde{\mathfrak{g}}/G\arrow[d, "q"] \\
M^{\mathrm{mero}} \arrow[r,"\mathrm{ev}_p"]
& \mathfrak{g}/G;
\end{tikzcd}
\end{equation*}
see \cite[Section 2.1]{Yun1}. This further yields the specialization
\begin{equation} \label{sp_diagram2}
\mathrm{ev}_p: M^{\mathrm{par}}_\zeta \rightarrow \widetilde{\CN}_\zeta/G \quad \rightsquigarrow \quad \mathrm{ev}_p: M^{\mathrm{par}}_0 \rightarrow \widetilde{\CN}_0/G.
\end{equation}
Finally, the pullback along $\mathrm{ev}_p$ of the $G$-equivariant version of the diagram (\ref{local_corr})
\begin{equation}\label{G_23}
    \begin{tikzcd}
\mathrm{Fl}/G \arrow[r, "\iota"] \arrow[d, "\pi "]
& \widetilde{\CN}_{0}/G \\
\mathrm{pt}/G
\end{tikzcd}
\end{equation}
recovers the diagram (\ref{corr}). Since all subsequent maps $\mathrm{ev}_p$ are obtained via base change from~(\ref{evaluation}), they are all smooth and surjective by Proposition \ref{prop3.2}.

Recall from Section \ref{Sec1.2}(a) that a general point $\eta \in T$ maps to a point $\zeta \in \overline{B}^{\mathrm{par}}$ (by abuse of notation), which forgets the order of the prescribed eigenvalues at $p$. The universal spectral curve $C_\eta \to W_\eta$ is obtained via the base change
\[
\begin{tikzcd}
    C_\eta \arrow[d, "\simeq"] \arrow[r, ""] & W_\eta \arrow[d, "\simeq"] \\
    C_\zeta \arrow[r, ""] & B_\zeta^{\mathrm{par}}
\end{tikzcd}
\]
which further induces the base change
\[
\begin{tikzcd}
    M^{\mathrm{par}}_\eta \arrow[d, "\simeq"] \arrow[r, ""] & W_\eta \arrow[d, "\simeq"] \\
    \overline{J}^e_{C_\zeta} \arrow[r, ""] & B_\zeta^{\mathrm{par}}.
\end{tikzcd}
\]
We also note that the fiber $M^{\mathrm{par}}_\zeta$ over $\zeta \in \overline{B}^{\mathrm{par}}$ consists of $n!$ copies of $\overline{J}^e_{C_\zeta}$ corresponding to the different orders of the eigenvalues at $p$, and that $M^{\mathrm{par}}_\eta$ is one of them.

\begin{prop}\label{prop3.4}
Assume that the self-correspondence
\[
\mathfrak{Z}_\eta \in \mathrm{Corr}_{W_\eta}^*(M^{\mathrm{par}}_\eta, M^{\mathrm{par}}_\eta)
\]
is pulled back from $\overline{J}^e_{C_\zeta} \times_{B^{\mathrm{par}}_\zeta}\overline{J}^e_{C_\zeta}$. Let 
\[
\mathfrak{Z}_0 \in \mathrm{Corr}^*_{W_0}(M^{\mathrm{par}}_0, M^{\mathrm{par}}_0)
\]
be the specialization of $\mathfrak{Z}_\eta$. Then \eqref{MM} commutes with $\mathfrak{Z}_0$, \emph{i.e.}, we have
\[
\left([\widetilde{M}_0] \circ {^\mathfrak{t}}[\widetilde{M}_0]\right) \circ \mathfrak{Z}_0 =  \mathfrak{Z}_0\circ \left([\widetilde{M}_0] \circ {^\mathfrak{t}}[\widetilde{M}_0]\right) \in \mathrm{Corr}^*_{W_0}(M^{\mathrm{par}}_0, M^{\mathrm{par}}_0). 
\]
\end{prop}

\begin{proof}
Since the diagram (\ref{corr}) is induced by the pullback of the diagram (\ref{G_23}) via the evaluation maps, we obtain that 
\begin{equation}\label{corrr}
\mathrm{ev}^* \left([\mathrm{Fl}/G] \circ {^\mathfrak{t}}[\mathrm{Fl}/G]\right) = [\widetilde{M}_0] \circ {^\mathfrak{t}}[\widetilde{M}_0].
\end{equation}
Here the evaluation map is given by
\[
\mathrm{ev}: M_0^{\mathrm{par}}\times_{W_0} M_0^{\mathrm{par}} \to \widetilde{\CN}_0/G \times \widetilde{\CN}_0/G.
\]
Lemma \ref{lem3.1}(a) then implies that the correspondence (\ref{corrr}) is the specialization via \eqref{sp_diagram2} of a self-correspondence of $M^{\mathrm{par}}_\zeta$ induced by $w_0 \in \mathfrak{S}_n$. 

For our purpose we need to compare the specialization via \eqref{sp_diagram} which is used to define the correspondence $\mathcal{Z}_0$, to the specialization via \eqref{sp_diagram2} which yields \eqref{corrr}. The former is the specialization coming from the family $\chi_p: M^{\mathrm{par}} \to \CA$ as in \eqref{chi_p}, which is smooth and compatible with intersection-theoretic operations. The latter can be rephrased as coming from the family given by the composition
\[
M^{\mathrm{par}} \xrightarrow{\chi_p} \CA \to \CA\sslash \mathfrak{S}_n \simeq \overline{B}^{\mathrm{par}}.
\]
This family is smooth in the following \emph{stacky} sense: we consider instead the composition of~$\chi_p$ with the map to the stack quotient
\begin{equation} \label{eq:stackfam}
M^{\mathrm{par}} \xrightarrow{\chi_p} \CA \to \CA/\mathfrak{S}_n.
\end{equation}
The specialization map for Chow groups with respect to the family \eqref{eq:stackfam} is again compatible with intersection-theoretic operations. Note that since we are dealing with quotient stacks, the required intersection theory, including the functorialities of the first six chapters of \cite{Ful} needed to define the specialization map, has been established in \cite{EG}. Further, by identifying the Chow groups/rings of a Deligne--Mumford stack and its coarse moduli space, we see that for cycles pulled back from correspondences for $\overline{J}^e_{C_\zeta}$ relative over $B_\zeta^{\mathrm{par}}$, the two specializations via \eqref{sp_diagram} and \eqref{sp_diagram2} only differ by a nonzero constant.

Finally, the pullback of any self-correspondence of $\overline{J}^e_{C_\zeta}$ to $M^{\mathrm{par}}_\zeta$ is invariant under the~$\mathfrak{S}_n$-action permuting the $n!$ copies of $\overline{J}^e_{C_\zeta}$, and hence commutes with the aforementioned self-correspondence of $M^{\mathrm{par}}_\zeta$ induced by \mbox{$\omega_0 \in \mathfrak{S}_n$}. The commutativity is preserved by specializing via the family \eqref{eq:stackfam}, which proves the proposition.
\end{proof}

Recall the morphism
\[
\Gamma: \mathrm{Corr}_{W_0}^0(M_0^{\mathrm{par}}, M^{\mathrm{par}}_0) \to \mathrm{Corr}_{B}^0(M_{n,d}, M_{n,d})
\]
from \eqref{eq:defGamma}. From now on, we use
\[
\mathfrak{p}^{\mathrm{par}}_k \in \mathrm{Corr}_{W_0}^0(M_0^{\mathrm{par}}, M^{\mathrm{par}}_0), \quad k=0,1, \cdots, 2(\dim M_0^{\mathrm{par}} - \dim W_0)
\]
to denote the projectors associated with the motivic perverse filtration of Theorem \ref{thm2.10}.

\begin{prop}\label{prop3.5}
The following statements hold.
\begin{enumerate}
    \item[(a)] We have
    \[    \Gamma([\Delta_{M_0^{\mathrm{par}}/W_0}]) = \varepsilon n!\cdot [\Delta_{M_{n,d}/B}].
    \]
    Here $\varepsilon  = \pm 1$ is as in Lemma \ref{lem3.1}(b).
    \item[(b)] We have the vanishing
    \[
    \Gamma(\mathfrak{p}_k^{\mathrm{par}}) = 0, \quad k < \binom{n}{2}.    \]
    \item[(c)] We have the stabilization 
    \[    \Gamma(\mathfrak{p}_k^{\mathrm{par}}) = [\Delta_{M_{n,d}/B}], \quad k\geq 2(\dim M_0^{\mathrm{par}} - \dim W_0) -\binom{n}{2}.    \]
  \item[(d)] We have the semi-orthogonality relation
  \[
 \Gamma(\mathfrak{p}^{\mathrm{par}}_l) \circ \Gamma(\mathfrak{p}^{\mathrm{par}}_k) = \varepsilon n! \cdot \Gamma(\mathfrak{p}^{\mathrm{par}}_k), \quad k < l.
 \]
\end{enumerate}
\end{prop}

\begin{proof}
(a) follows from pulling back the local statement of Lemma \ref{lem3.1}(b). (b) and (c) follow from Lemma \ref{lem1.2} and the construction of $\mathfrak{p}^{\mathrm{par}}_k$ in terms of the specialization of algebraic cycles of the form (\ref{motivic_p_k}). (d) is a consequence of part (a) and Proposition \ref{prop3.4}:
\begin{align*}
     \Gamma(\mathfrak{p}^{\mathrm{par}}_l) \circ \Gamma(\mathfrak{p}^{\mathrm{par}}_k) & =  {^\mathfrak{t}[\widetilde{M}_0]}\circ\mathfrak{p}^{\mathrm{par}}_l
     \circ \left([\widetilde{M}_0] \circ {^\mathfrak{t}[\widetilde{M}_0]} \right)\circ\mathfrak{p}^{\mathrm{par}}_k
     \circ [\widetilde{M}_0] \\
     & = \left({^\mathfrak{t}[\widetilde{M}_0]}  \circ [\widetilde{M}_0] \right) \circ {^\mathfrak{t}[\widetilde{M}_0]}    \circ \left(\mathfrak{p}^{\mathrm{par}}_l  \circ\mathfrak{p}^{\mathrm{par}}_k \right)
     \circ [\widetilde{M}_0] \\
     & = \varepsilon n!\cdot  {^\mathfrak{t}[\widetilde{M}_0]}    \circ \mathfrak{p}^{\mathrm{par}}_k \circ [\widetilde{M}_0]    \\
     & = \varepsilon n! \cdot \Gamma(\mathfrak{p}^{\mathrm{par}}_k).
\end{align*}
Here we have used that every $\mathfrak{p}_k^{\mathrm{par}}$ satisfies the assumption of $\mathfrak{Z}_0$ in Proposition \ref{prop3.4}, and we applied (a) to deduce the third identity.
\end{proof}

By a direct dimension calculation, we have
\[
(\dim M_0^{\mathrm{par}} - \dim W_0) - \binom{n}{2} = d_f.
\]
Therefore, after rescaling and relabelling as
\[
\mathfrak{p}_k:= \frac{\varepsilon}{n!} \cdot \Gamma\left(\mathfrak{p}^{\mathrm{par}}_{k-\binom{n}{2}}\right), 
\]
we obtain the projectors
\begin{equation}\label{final_perv}
\mathfrak{p}_k \in \mathrm{Corr}_B^0(M_{n,d}, M_{n,d}), \quad k=0,1,\cdots, 2d_f.
\end{equation}
In Section \ref{Section4}, we will prove that (\ref{final_perv}) forms a motivic perverse filtration; as for Theorem \ref{thm2.10}, we only need to show that the homological realization of $\mathfrak{p}_k$ recovers the perverse truncation functor.

\section{Sheaf-theoretic operations}\label{Section4}

The purpose of this section is to prove a general specialization result for the decomposition theorem (Proposition \ref{prop4.1}). Then we apply it to the Hitchin systems we studied in the previous sections and complete the proofs of all the main theorems.


\subsection{Specializations}\label{Sec4.1}

In this section, we finish the proof of Theorem \ref{thm2.10} by addressing the compatibility of the Chow and sheaf-theoretic specialization maps, and their relationship with the perverse truncation. 

Our setup is the following: let $T \simeq \mathbb{A}^1$ and let $\pi: B \to T$ be a smooth morphism. Consider two proper morphisms $f: X \to B$, $g: Y \to B$ such that the compositions $\pi \circ f: X \to T$, $\pi \circ g: Y \to T$ are smooth:
\[
\begin{tikzcd}
X \arrow[r, "f"] \arrow[rd, ""] & B \arrow[d, "\pi"] & Y \arrow[l, "g", swap] \arrow[ld, ""] \\
& T.
\end{tikzcd}
\]
We think of $f: X \to B$, $g: Y \to B$ as two families of proper morphisms over $T$.

Suppose we are given over a nonempty open subset $T^\circ \subset T$ a correspondence
\begin{equation} \label{eq:defz}
\FZ \in \Corr^*_{B_{T^\circ}}(X_{T^\circ}, Y_{T^\circ}) = \mathrm{CH}_*(X_{T^\circ} \times_{B_{T^\circ}} Y_{T^\circ}),
\end{equation}
which puts together a family of correspondences $\FZ_\eta$ for $\eta \in T^\circ$:
\[
\FZ_\eta \in \Corr^*_{B_\eta}(X_\eta, Y_\eta) = \mathrm{CH}_*(X_\eta \times_{B_\eta} Y_\eta).
\]
Let $0 \in T$ be the special point. Then by the specialization map for Chow groups~via
\[
X_\eta \to B_\eta \quad \rightsquigarrow \quad X_0 \to B_0,
\]
we obtain a correspondence
\[
\FZ_0 \in \Corr^*_{B_\eta}(X_\eta, Y_\eta) = \mathrm{CH}_*(X_0 \times_{B_0} Y_0).
\]
Also recall from \cite[Lemma 2.21]{CH} that the cycle class $\mathrm{cl}(\FZ_\eta) \in H^{\mathrm{BM}}_{2*}(X_\eta \times_{B_\eta} Y_\eta, \BQ)$ induces a morphism in $D^\mathrm{b}_\mathrm{c}(B_\eta)$:
\begin{equation} \label{eq:clceta}
\mathrm{cl}(\FZ_\eta): Rf_{\eta*}\BQ_{X_\eta} \to Rg_{\eta*}\BQ_{Y_\eta}[*],
\end{equation}
where we ignore the precise shifting index. Similarly, we have a morphism in $D^\mathrm{b}_\mathrm{c}(B_0)$:
\begin{equation} \label{eq:clc0}
\mathrm{cl}(\FZ_0): Rf_{0*}\BQ_{X_0} \to Rg_{0*}\BQ_{Y_0}[*].
\end{equation}

\begin{prop} \label{prop4.1}
Let $\FZ_\eta, \FZ_0$ be as above. For $k \in \BZ$, if
\begin{equation} \label{eq:phketa}
{^\Fp \CH}^k(\mathrm{cl}(\FZ_\eta)) : {^\Fp \CH}^k(Rf_{\eta*}\BQ_{X_\eta}) \to {^\Fp \CH}^k(Rg_{\eta*}\BQ_{Y_\eta}[*])
\end{equation}
is an isomorphism (resp.~zero), then the same holds for
\begin{equation} \label{eq:phk0}
{^\Fp \CH}^k(\mathrm{cl}(\FZ_0)) : {^\Fp \CH}^k(Rf_{0*}\BQ_{X_0}) \to {^\Fp \CH}^k(Rg_{0*}\BQ_{Y_0}[*]).
\end{equation}
\end{prop}

\begin{proof}
Let $\overline{\FZ} \in \mathrm{CH}_*(X \times_B Y)$ be any  extension of $\FZ \in \mathrm{CH}_*(X_{T^\circ} \times_{B_{T^\circ}} Y_{T^\circ})$ in \eqref{eq:defz}, \emph{e.g.}~by taking the Zariski closure. The cycle class $\mathrm{cl}(\overline{\FZ}) \in H^{\mathrm{BM}}_{2*}(X \times_B Y, \BQ)$ induces a morphism in~$D^\mathrm{b}_\mathrm{c}(B)$:
\begin{equation} \label{eq:clc}
\mathrm{cl}(\overline{\FZ}) : Rf_*\mathbb{Q}_X \to Rg_*\mathbb{Q}_Y[*].
\end{equation}
Then by construction and proper base change, the morphism $\mathrm{cl}(\FZ_0)$ in \eqref{eq:clc0} is obtained by applying~$i^*$ to \eqref{eq:clc} for the inclusion $i: B_0 \hookrightarrow B$:
\[
\mathrm{cl}(\FZ_0) = i^*\mathrm{cl}(\overline{\FZ}): Rf_{0*}\BQ_{X_0} \to Rg_{0*}\BQ_{Y_0}[*].
\]

We give an alternative construction of \eqref{eq:clc0} via nearby and vanishing cycles following \cite{dCMa, dC2}. Recall the distinguished triangle
\begin{equation} \label{eq:can}
i^*[-1] \to \psi_B[-1] \xrightarrow{\mathrm{can}} \phi_B \xrightarrow{+1}
\end{equation}
where $\psi_B, \phi_B : D^\mathrm{b}_\mathrm{c}(B) \to D^\mathrm{b}_\mathrm{c}(B_0)$ are the nearby and vanishing cycle functors respectively. Since $X$ is smooth over~$T$, we have $\phi_X\mathbb{Q}_X \simeq 0 \in D^\mathrm{b}_\mathrm{c}(X_0)$ and by proper base change for $\phi$,
\[
\phi_B Rf_*\mathbb{Q}_X \simeq Rf_{0*}\phi_X\mathbb{Q}_X \simeq 0 \in D^\mathrm{b}_\mathrm{c}(B_0).
\]
Similarly, we have
\[
\phi_B Rg_*\mathbb{Q}_Y \simeq Rg_{0*}\phi_Y\mathbb{Q}_Y \simeq 0 \in D^\mathrm{b}_\mathrm{c}(B_0).
\]
We now apply the distinguished triangle \eqref{eq:can} to the morphism $\mathrm{cl}(\overline{\FZ})$ in \eqref{eq:clc}, which yields a commutative diagram
\begin{equation} \label{eq:nophi}
\begin{tikzcd}
Rf_{0*}\BQ_{X_0} \arrow[r, "\mathrm{cl}(\FZ_0)"] \arrow[d, "\simeq"] & Rg_{0*}\BQ_{Y_0}[*] \arrow[d, "\simeq"] \\
\psi_BRf_*\BQ_X \arrow[r, "\psi_B\mathrm{cl}(\overline{\FZ})"] & \psi_BRg_*\BQ_Y[*].
\end{tikzcd}
\end{equation}
Here both columns are isomorphisms since the vanishing cycles are trivial. Moreover, since~$\psi_B$ is the \emph{nearby cycle} functor, the bottom morphism $\psi_B\mathrm{cl}(\overline{\FZ})$ is fully determined by the morphism~$\mathrm{cl}(\FZ_\eta)$ in~\eqref{eq:clceta}.

Finally, we recall the fact that $\psi_B[-1]$ is perverse $t$-exact, and hence commutes with taking perverse cohomology. Therefore, if \eqref{eq:phketa} is an isomorphism (resp.~zero), then
\[
{^\Fp \CH}^k(\psi_B\mathrm{cl}(\overline{\FZ})) : {^\Fp \CH}^k(\psi_B Rf_*\BQ_X) \to {^\Fp \CH}^k(\psi_B Rg_*\BQ_Y[*])
\]
is also an isomorphism (resp.~zero). By the diagram \eqref{eq:nophi} the same holds for \eqref{eq:phk0}.
\end{proof}

A useful consequence of Proposition \ref{prop4.1} is that the validity of the Corti--Hanamura motivic decomposition conjecture \cite{CH} is preserved under specialization. This also completes the proof of Theorem \ref{thm2.10}.

\begin{cor}\label{cor4.2}
Let $X \xrightarrow{f} B \xrightarrow{\pi} T$ be as above, with $f$ proper and $\pi, \pi \circ f$ smooth. If the motivic decomposition conjecture holds for $X_\eta \to B_\eta$, then the specializations of the projectors provide a motivic decomposition for $X_0 \to B_0$.
\end{cor}

\begin{proof}
The (semi-)orthogonal projectors of $X_\eta \to B_\eta$ specialize immediately to (semi-)orthog\-onal projectors of $X_0 \to B_0$, while the relative diagonal class $[\Delta_{X_\eta/B_\eta}]$ specializes to $[\Delta_{X_0/B_0}]$. Now Proposition~\ref{prop4.1} guarantees that the homological realization of the specialized projectors recovers the decomposition theorem for $X_0 \to B_0$.
\end{proof}

\subsection{Correspondences}

For our purpose, we connect the decomposition theorem associated with the Hitchin maps:
\[
h_0: M^{\mathrm{par}}_0 \to W_0, \quad f: M_{n,d} \to B \subset W_0.
\]
The diagram (\ref{corr}) yields sheaf-theoretic correspondences 
\begin{equation}\label{Sheaf_corr}
[\widetilde{M}_0]: Rh_{0*}\BQ_{M^{\mathrm{par}}_0} \to  Rf_* \BQ_{M_{n,d}} \left[- 2\binom{n}{2} \right], \quad {^\mathfrak{t}}[\widetilde{M}_0]:  Rf_* \BQ_{M_{n,d}}\left[ -2\binom{n}{2} \right] \to Rh_{0*}\BQ_{M^{\mathrm{par}}_0} 
\end{equation}
which take places in the derived category $D^\mathrm{b}_\mathrm{c}(W_0)$.

The following lemma is a global version of Lemma \ref{lem3.1}(c).

\begin{lem}\label{lem4.1}
There is a natural Weyl group $\mathfrak{S}_n$-action on the object
\[
Rh_{0*}\BQ_{M^{\mathrm{par}}_0} \in D^\mathrm{b}_\mathrm{c}(W_0);
\]
the anti-invariant part is 
\[
\left(Rh_{0*}\BQ_{M^{\mathrm{par}}_0}\right)^{\mathrm{sgn}} = Rf_* \BQ_{M_{n,d}}\left[ -2\binom{n}{2} \right] \in D^\mathrm{b}_\mathrm{c}(W_0).
\]
Furthermore, under this isomorphism the correspondences \eqref{Sheaf_corr} are projecting to the anti-invariant part and the natural inclusion from the anti-invariant part respectively.
\end{lem}

\begin{proof}
    By Proposition \ref{prop3.2}, the diagram (\ref{corr}) is the pullback of the local diagram (\ref{local_corr}) along smooth morphisms $\mathrm{ev}_p$. Therefore, Lemma \ref{lem4.1} follow from pulling back the $G$-equivariant version of Lemma \ref{lem3.1}(c).
\end{proof}

\subsection{Proofs of Theorems \ref{thm1}, \ref{motivic_lefschetz}, \ref{thm2}, \ref{thm3}}

In this section, we complete the proofs of all the main theorems.

\subsubsection{Proof of Theorem \ref{thm1}}\label{sec4.3.1}

We start with the proof of Theorem \ref{thm1}. By Proposition \ref{prop2.1}, it suffices to show that the correspondences we constructed in (\ref{final_perv}) form a motivic perverse filtration. By Proposition \ref{prop3.5}, it remains to check that the homological realization
\[
\mathfrak{p}_k (Rf_* \BQ_{M_{n,d}}) \rightarrow Rf_* \BQ_{M_{n,d}}
\]
recovers the perverse truncation
 \[
    ^{\mathfrak{p}}\tau_{\leq k+ \dim {B}} Rf_*\BQ_{M_{n,d}} \rightarrow  Rf_*\BQ_{M_{n,d}}.
\]
We know that this statement holds for $h_0: M^{\mathrm{par}}_0 \to W_0$ and $\mathfrak{p}_k^{\mathrm{par}}$ by Theorem \ref{thm2.10} (proven in Section \ref{Sec4.1}). Now, we consider the correspondence of $\mathfrak{p}_k$ given by the composition:
\[
\mathfrak{p}_k: Rf_* \BQ_{M_{n,d}} \xrightarrow{[\widetilde{M}_0]} Rh_{0*} \BQ_{M_0^{\mathrm{par}}} \left[2\binom{n}{2} \right] \xrightarrow{\mathfrak{p}^{\mathrm{par}}_k}  Rh_{0*} \BQ_{M_0^{\mathrm{par}}} \left[2\binom{n}{2} \right]  \xrightarrow{{^\mathfrak{t}[\widetilde{M}_0]}} Rf_* \BQ_{M_{n,d}}.
\]
By Lemma \ref{lem4.1}, we have a direct sum decomposition induced by the $\mathfrak{S}_n$-action:
\begin{equation}\label{isotypic}
Rh_{0*} \BQ_{M_0^{\mathrm{par}}} \left[2\binom{n}{2} \right] = (Rh_{0*} \BQ_{M_0^{\mathrm{par}}})^{\mathrm{sgn}} \left[2\binom{n}{2} \right] \oplus \mathrm{others},
\end{equation}
and applying $\mathfrak{p}_k$ to $Rf_* \BQ_{M_{n,d}}$ is induced by applying $\mathfrak{p}_k^{\mathrm{par}}$ on the first factor above. In particular, it coincides with the perverse truncation. This completes the proof. \qed

\subsubsection{ Proof of Theorem \ref{motivic_lefschetz}}

We first note the following result of \cite[Proposition 6.4]{ACLS}.\label{sec4.3.2}

\medskip
{\noindent \bf Fact 1.} The compactified Jacobian fibration
\[
h_\eta: M^{\mathrm{par}}_\eta \to W_\eta
\]
of Section \ref{Sec1.2}(a) satisfies the relative Lefschetz standard conjecture. 

\medskip
More precisely, this is due to the resolutions of the Lefschetz standard conjecture for abelian varieties \cite{Lieb, Kun} and the fact that the decomposition theorem associated with a compactified Jacobian fibration associated with integral locally planar curves has \emph{full support} by the Ng\^{o} support theorem \cite{Ngo}.

Our strategy is to reduce the case of $f: M_{n,d} \to B$ to $h_\eta: M^{\mathrm{par}}_\eta \to W_\eta$. Let $\sigma \in H^2(M_{n,d}, \BQ)$ be a relative ample class. By \cite{Markman}, we can write $\sigma$ in terms of tautological classes as follows. We first recall that we have a canonical isomorphism
\begin{equation}\label{decom_1}
H^2(M_{n,d}, \BQ) = H^2(\widehat{M}_{n,d}, \BQ) \oplus H^2(J_\Sigma, \BQ)
\end{equation}
where $\widehat{M}_{n,d}$ is the $\mathrm{PGL}_n$-Hitchin moduli space and $J_\Sigma$ is the Jacobian variety of the curve $\Sigma$; see \cite[Section 2.4]{dCHM1}. The first summand on the right-hand side of (\ref{decom_1}) spans a $1$-dimensional subspace, generated by the relative ample tautological class given by the K\"unneth component of normalized second Chern character $\widetilde{\mathrm{ch}}_2(\CU)$ with respect to the point class $[\mathrm{pt}]\in H^2(\Sigma, \BQ)$ on the curve. This class can be lifted to the ($\BQ$-)divisor class
\[
c_2([\Sigma]): = \pi_{M*}\left(\pi_\Sigma^* [\Sigma] \cup \widetilde{\mathrm{ch}}_2(\CU) \right) \in \mathrm{Pic}(M_{n,d})_\BQ.
\]
Therefore, under the decomposition (\ref{decom_1}), we may present the relative ample class $\sigma$ as
\begin{equation}\label{taut_pre}
\sigma = c_2([\Sigma]) + \theta 
\end{equation}
where $\theta$ is induced by an ample divisor on $J_\Sigma$. The right-hand side is the pullback of a divisor on $M^{\mathrm{mero}}$ with the same tautological presentation (\ref{taut_pre}). As in Section \ref{Sec2_last}, pulling back this class yields relative ample classes on each variety $M^{\mathrm{par}}_\eta$ with $\eta \in T$.

In conclusion, we obtain the following fact.

\medskip

{\noindent \bf Fact 2.} There exists a divisor $\sigma_0^{\mathrm{par}} \in \mathrm{Pic}(M_0^{\mathrm{par}})_\BQ$ satisfying
\begin{equation} \label{eq:piiota}
\pi^* \sigma = \iota^* \sigma_0^{\mathrm{par}}
\end{equation}
where $\iota, \pi$ are defined in the diagram (\ref{corr}). Moreover, this divisor $\sigma_0^{\mathrm{par}}$ is obtained as the specialization of a relative ample divisor
\[
\sigma_\eta^{\mathrm{par}} \in \mathrm{Pic}(M^\mathrm{par}_\eta)_{\BQ}
\]
for general $\eta$.
\medskip

Note that in general the cup-product with respect to an algebraic cycle $E$ is induced by a correspondence given by the $\Delta_*E$, where $\Delta$ is the (relative) diagonal. Now we consider the correspondences
\[
\Delta_*{\sigma^i} \in \mathrm{Corr}^*_{B}(M_{n,d}, M_{n,d}), \quad \Delta_*{(\sigma^{\mathrm{par}}_0)^i} \in \mathrm{Corr}_{W_0}^*(M_0^{\mathrm{par}}, M_0^{\mathrm{par}}).
\]

\medskip

{\noindent \bf Fact 3.} We have
\[
\Gamma\left(\Delta_*{(\sigma^{\mathrm{par}}_0)^i} \right) = \varepsilon n!\cdot \Delta_*{\sigma^i} \in \mathrm{Corr}^*_B(M_{n,d}, M_{n,d}).
\]
In particular, after the homological realization we obtain the commutative diagram
\begin{equation*}
\begin{tikzcd}
Rf_* \BQ_{M_{n,d}} \arrow[r, "\varepsilon n!\cdot(\cup \sigma^i)"] \arrow[d, "{[\widetilde{M}_0]}"]
& Rf_* \BQ_{M_{n,d}}[2i] \\
Rh_{0*}\BQ_{M^{\mathrm{par}}_0}\left[2\binom{n}{2} \right] \arrow[r,"\cup (\sigma_0^{\mathrm{par}})^i"]
& Rh_{0*}\BQ_{M^{\mathrm{par}}_0}\left[2\binom{n}{2} +2i\right]\arrow[u, "{{^\mathfrak{t}}[\widetilde{M}_0]}"].
\end{tikzcd}
\end{equation*}

\medskip

\begin{proof}[Proof of Fact 3]
It suffices to show the following relation of correspondences
\begin{equation} \label{eq:commu}
{^\Ft[\widetilde{M}_0]} \circ \Delta_*(\sigma_0^{\mathrm{par}})^i = \Delta_*\sigma^i \circ {^\Ft[\widetilde{M}_0]} \in \Corr^*_{W_0}(M_0^{\mathrm{par}}, M_{n, d}).
\end{equation}
Then by Proposition \ref{prop3.5}(a), we have
\[
\Gamma(\Delta_*(\sigma_0^{\mathrm{par}})^i) = {^\Ft[\widetilde{M}_0]} \circ \Delta_*(\sigma_0^{\mathrm{par}})^i \circ [\widetilde{M}_0] = \Delta_*\sigma^i \circ {^\Ft[\widetilde{M}_0]} \circ [\widetilde{M}_0] = \varepsilon n! \cdot \Delta_*\sigma^i,
\]
which proves Fact 3. To see the relation \eqref{eq:commu}, we expand the left-hand side
\begin{align*}
{^\Ft[\widetilde{M}_0]} \circ \Delta_*(\sigma_0^{\mathrm{par}})^i & = p_{13*}\delta^!\left(\Delta_*(\sigma_0^{\mathrm{par}})^i \times {^\Ft[\widetilde{M}_0]}\right) \\
& = p_{13*}\delta^!\left(\left(q_2^*(\sigma_0^\mathrm{par})^i \cap [\Delta_{M_0^{\mathrm{par}}/W_0}]\right) \times {^\Ft[\widetilde{M}_0]}\right) \\
& = p_{13*}\delta^!\left([\Delta_{M_0^{\mathrm{par}}/W_0}] \times \left(r_1^*(\sigma_0^\mathrm{par})^i \cap {^\Ft[\widetilde{M}_0]}\right)\right) \\
& = r_1^*(\sigma_0^\mathrm{par})^i \cap {^\Ft[\widetilde{M}_0]} = {^\Ft(\iota^*(\sigma_0^\mathrm{par})^i \cap [\widetilde{M}_0])}.
\end{align*}
Here $\delta^!$ is the refined Gysin pullback with respect to the regular (absolute diagonal) embedding~$\Delta: M_0^{\mathrm{par}} \to M_0^{\mathrm{par}} \times M_0^{\mathrm{par}}$, and $p_{13}: M_0^{\mathrm{par}} \times_{W_0} M_0^{\mathrm{par}} \times_{W_0} M_{n, d} \to M_0^{\mathrm{par}} \times_{W_0} M_{n, d}$, $q_2: M_0^{\mathrm{par}} \times_{W_0} M_0^{\mathrm{par}} \to M_0^{\mathrm{par}}$, and $r_1: M_0^{\mathrm{par}} \times_{W_0} M_{n, d} \to M_0^{\mathrm{par}}$ are the natural projections. Similarly, we compute the right-hand side of \eqref{eq:commu}:
\begin{align*}
\Delta_*\sigma^i \circ {^\Ft[\widetilde{M}_0]} & = p_{13*}\delta^!\left({^\Ft[\widetilde{M}_0]} \times \Delta_*\sigma^i\right) \\
& = p_{13*}\delta^!\left({^\Ft[\widetilde{M}_0]} \times \left(r_1^*\sigma^i \cap [\Delta_{M_{n, d}/B}]\right)\right) \\
& = p_{13*}\delta^!\left(\left(q_2^*\sigma^i \cap {^\Ft[\widetilde{M}_0]}\right) \times [\Delta_{M_{n, d}/B}]\right) \\
& = q_2^*\sigma^i \cap {^\Ft[\widetilde{M}_0]} = {^\Ft(\pi^*\sigma^i \cap [\widetilde{M}_0])}.
\end{align*}
This time $\delta^!$ is the refined Gysin pullback with respect to $\Delta: M_{n, d} \to M_{n, d} \times M_{n, d}$, and $p_{13}, q_2, r_1$ are the corresponding natural projections. Hence \eqref{eq:commu} follows from \eqref{eq:piiota}.
\end{proof}

Now we construct the cycle $\mathfrak{Z}_{\sigma,i}$ for the relative ample class $\sigma$ on $M_{n,d}$. We start with the correspondence as in Fact 1 for $M^{\mathrm{par}}_\eta$: 
\begin{equation} \label{zpareta}
\mathfrak{Z}_{\sigma^{\mathrm{par}}_{\eta},i} \in \mathrm{Corr}_{W_\eta}^*( M^{\mathrm{par}}_\eta, M^{\mathrm{par}}_\eta)
\end{equation}
which gives the inverse of the Lefschetz symmetry induced by $(\sigma^{\mathrm{par}}_{\eta})^i$ for general $\eta$.

By Proposition \ref{prop4.1}, its specialization
\begin{equation}\label{zpar}
\mathfrak{Z}_{\sigma^{\mathrm{par}}_{0},i} \in \mathrm{Corr}_{W_0}^*( M^{\mathrm{par}}_0, M^{\mathrm{par}}_0)
\end{equation}
gives the inverse of the Lefschetz symmetry induced by $(\sigma^{\mathrm{par}}_{0})^i$, which respects the decomposition given by the $\mathfrak{S}_n$-structure since \eqref{zpareta} is pulled back from a self-correspondence of~$\overline{J}^e_{C_{\zeta}}$ relative over $B_\zeta^{\mathrm{par}}$. In particular, the correspondence (\ref{zpar}) induces an isomorphism on the perverse cohomology sheaves of the first summand of (\ref{isotypic}).

Finally, we set
\[
\mathfrak{Z}_{\sigma,i}: = \Gamma \left( \mathfrak{Z}_{\sigma^{\mathrm{par}}_{0},i}\right) \in \mathrm{Corr}_{B}^*(M_{n,d}, M_{n,d}).
\]
As in the proof of Theorem \ref{thm1} in Section \ref{sec4.3.1}, the diagram of Fact 3 together with Lemma~\ref{lem4.1} implies that the sheaf-theoretic operator $\cup \sigma^{i}$ coincides with the restriction of the operator~$\cup (\sigma^{\mathrm{par}}_0)^i$ to the anti-invariant component with respect to the $\mathfrak{S}_n$-equivariant structure. Since the inverse operator (\ref{zpar}) of $\cup (\sigma^{\mathrm{par}}_0)^i$ also respects the decomposition (\ref{isotypic}), we conclude that the homological realization of $\mathfrak{Z}_{\sigma, i}$  gives the inverse of  $\cup \sigma^{i}$ as desired.
\qed


\subsubsection{Proof of Theorem \ref{thm2}}\label{sec4.3.3}
We consider the diagram 
\begin{equation}\label{corr_Sigma}
    \begin{tikzcd}
\Sigma \times \widetilde{M}_0 \arrow[r, "\iota"] \arrow[d, "\pi"]
&\Sigma \times M_0^{\mathrm{par}} \\
\Sigma\times M_{n,d} 
\end{tikzcd}
\end{equation}
given by the base change of (\ref{corr}). For convenience, we use the same notation for this diagram as for (\ref{corr}); \emph{e.g.}~we still use $[\widetilde{M}_0]$ to denote the correspondence from $\Sigma \times M_{n,d}$ to $\Sigma \times M_0^{\mathrm{par}}$.

By the construction of the normalized Chern character, we have
\[
\pi^*\widetilde{\mathrm{ch}}(\CU) = \iota^* \widetilde{\mathrm{ch}}(\CU).
\]
Here the normalized Chern character on the left-hand side lies over $\Sigma \times M_{n,d}$ and the normalized Chern character on the right-hand side lies over $\Sigma \times M_{0}^{\mathrm{par}}$; their match after pulling back to~$\Sigma \times \widetilde{M}_0$ follows from the fact that both normalized Chern characters are pulled back from~$\Sigma \times M^{\mathrm{mero}}$. By the same argument as for Fact 3 in Section \ref{sec4.3.2}, we have
\begin{equation}\label{match1}
\Gamma\left(\Delta_* \widetilde{\mathrm{ch}}_k(\CU)\right) = \varepsilon n!\cdot \Delta_* \widetilde{\mathrm{ch}}_k(\CU)\in  \mathrm{Corr}^*_{\Sigma\times B}(\Sigma \times M_{n,d}, \Sigma \times M_{n,d})
\end{equation}
where the two sides concern the normalized Chern characters on $\Sigma\times M^{\mathrm{par}}_0$ and $\Sigma \times M_{n,d}$ respectively. Recall the projectors $\mathfrak{p}_k$ from (\ref{corr}) and we use the same notation to denote the pullback of these projectors to the varieties in (\ref{corr_Sigma}). We define
\[
\mathfrak{q}_{k+1} := \Delta_{\Sigma\times M_{n,d}/\Sigma\times B} - \mathfrak{p}_k = \frac{\varepsilon}{n!}\cdot \Gamma\left(\Delta_{\Sigma\times M_0^{\mathrm{par}}/\Sigma\times B} - \mathfrak{p}^{\mathrm{par}}_{k-\binom{n}{2}}\right),
\]
where the second identity is a consequence of Proposition \ref{prop3.5}(a). 

For our purpose, we want to prove the relation:
\[
\mathfrak{q}_{i + k+1} \circ \Delta_* \widetilde{\mathrm{ch}}_k(\CU) \circ \mathfrak{p}_i = 0;
\]
using (\ref{match1}) this can be rewritten as
\begin{equation}\label{relation_100}
\Gamma\left(\mathfrak{q}^{\mathrm{par}}_{i + k+1-\binom{n}{2}}\right) \circ \Gamma\left(\Delta_* \widetilde{\mathrm{ch}}_k(\CU)\right)\circ \Gamma\left( \mathfrak{p}^{\mathrm{par}}_{i-\binom{n}{2}} \right) = 0.
\end{equation}
Since all the correspondences
\[
\mathfrak{q}^{\mathrm{par}}_{i + k+1-\binom{n}{2}}, \quad \Delta_* \widetilde{\mathrm{ch}}_k(\CU), \quad \mathfrak{p}^{\mathrm{par}}_{i-\binom{n}{2}} \in \mathrm{Corr}^*_{\Sigma\times W_0}(\Sigma\times M_0^{\mathrm{par}}, \Sigma \times M_0^{\mathrm{par}})
\]
are specializations of cycles pulled back from correspondences for $\overline{J}^e_{C_\zeta}$ relative over $B_\zeta^{\mathrm{par}}$, by Proposition~\ref{prop3.4} the relation (\ref{relation_100}) follows from the relation
\[
\mathfrak{q}^{\mathrm{par}}_{i + k+1-\binom{n}{2}} \circ \Delta_* \widetilde{\mathrm{ch}}_k(\CU)\circ \mathfrak{p}^{\mathrm{par}}_{i-\binom{n}{2}} =0 \in \mathrm{Corr}^*_{\Sigma\times W_0}(\Sigma\times M_0^{\mathrm{par}}, \Sigma \times M_0^{\mathrm{par}}).
\]
This is given by Theorem \ref{thm2.11}. The proof is completed. \qed

\subsubsection{Proof of Theorem \ref{thm3}}

We denote by
\[
\mathfrak{C} \in \mathrm{Corr}_{W_0}^0(M^{\mathrm{par}}_0, {M^{\mathrm{par}}_{0}}'), \quad \mathfrak{C}' \in \mathrm{Corr}_{W_0}^0({M_{0}^{\mathrm{par}}}', M_{0}^{\mathrm{par}})
\]
the correspondences constructed in Theorem \ref{thm2.12} inducing an isomorphism
\[
h(M_0^{\mathrm{par}}) \simeq h({M_0^{\mathrm{par}}}').
\]
Then an identical argument as above (using Proposition \ref{prop3.4}) implies that
\[
\overline{\mathfrak{C}}:= \frac{\varepsilon}{n!} \cdot \mathfrak{C} \in \mathrm{Corr}_B^0(M_{n,d}, M_{n,d'}), \quad \overline{\mathfrak{C}}':= \frac{\varepsilon}{n!} \cdot \mathfrak{C}' \in \mathrm{Corr}_B^{0}(M_{n,d'}, M_{n,d})
\]
induces an isomorphism as desired.  \qed

\subsubsection{Remarks on the twisted case} \label{sec4.3.5}

As in Remark \ref{remark1.1}, one can also consider the $\mathfrak{D}$-twisted case with $\mathfrak{D} >0$. Here we discuss briefly the possible strategies to prove the main theorems in the~$\mathfrak{D}$-twisted case.

The decomposition theorem associated with $f: M^{\mathfrak{D}}_{n,d} \to B^{\mathfrak{D}}$ has full support; see Remark~\ref{remark1.1} and \cite{CL}. In this case, Theorem \ref{motivic_lefschetz} follows directly from \cite[Proposition 6.4]{ACLS} as Fact 1 above, and the argument of Section \ref{sec4.3.3} is not needed.

On the other hand, due to the failure of the $\mathfrak{D}$-twisted version of Lemma \ref{lem1.20}, for general~$\eta \in T$ we can only guarantee that all the spectral curves are reduced, but there may exist non-integral ones. Therefore, the extension of Arinkin's Fourier--Mukai transform over reduced locally planar curves \cite{MRV1, MRV2} may be needed to construct the projectors for the motivic decomposition.

\section{Questions and conjectures}

In this section, we discuss some conjectures and open questions concerning algebraic cycles for the Hitchin system.

\subsection{Multiplicativity}

A central question raised by de Cataldo, Hausel, and Migliorini in accompany with their discovery of the P=W conjecture (see \cite[Introduction]{dCHM1}) is the \emph{multiplicativity} of the perverse filtration:
\begin{equation}\label{quest1}
\cup: P_kH^*(M_{n,d}, \BQ) \times P_{l}H^*(M_{n,d}, \BQ) \rightarrow P_{k+l} H^*(M_{n,d}, \BQ).
\end{equation}
This phenomenon now is a consequence of the resolutions of the P=W conjecture, whose proofs rely heavily on the tautological generation result of Markman \cite{Markman}. 

In general, the multiplicativity holds for Leray filtrations associated with proper morphisms, but fails for the perverse filtration. The main reason is that the (sheaf-theoretic) cup-product does not interact nicely with the perverse truncation functor in general.

For the Hitchin system, we expect that the sheaf-theoretic enhancement of (\ref{quest1}) holds.

\begin{conj}\label{conj1}
The perverse truncation functor interacts with cup-product as follows
\[
\cup: {^\mathfrak{p}}\tau_{\leq k} Rf_* \BQ_{M_{n,d}} \times {^\mathfrak{p}}\tau_{\leq l} Rf_* \BQ_{M_{n,d}} \to {^\mathfrak{p}}\tau_{\leq k+l- (\dim M_{n, d} - d_f)} Rf_* \BQ_{M_{n,d}}.
\]
\end{conj}

By the proof of Theorem \ref{thm1}, the perverse filtration and its sheaf-theoretic enhancement are governed by the algebraic cycles $\mathfrak{p}_k, \mathfrak{q}_k$. We propose the following conjecture which further strengthens Conjecture \ref{conj1}; this also provides an explanation of the multiplicativity phenomenon in terms of \emph{tautological relations} concerning $\mathfrak{p}_k, \mathfrak{q}_k$.

\begin{conj}\label{conj2}
The motivic perverse filtration (\ref{Motivic_P}) is multiplicative with respect to cup-product: 
\[
\cup: P_k h(M_{n,d}) \times P_lh(M_{n,d}) \to P_{k+l}h(M_{n,d}),
\]
\emph{i.e.}, the following relations hold
\[
\Fq_{k+l+1} \circ [\Delta^{\mathrm{sm}}_{M_{n, d}/B}]\circ (\Fp_k \times \Fp_l) =0 \in \mathrm{CH}_*(M_{n,d} \times_B M_{n,d} \times_B M_{n,d} ).
\]
Here $\Delta^{\mathrm{sm}}_{M_{n, d}/B}$ is the relative small diagonal in the relative triple product.
\end{conj}

By Theorem \ref{thm2.3}(b), Conjecture \ref{conj2} holds for the parabolic Hitchin system $h_\eta: M^{\mathrm{par}}_\eta \to W_\eta$ for general $\eta$ or the elliptic part of the Hitchin system $f^{\mathrm{ell}}: M^{\mathrm{ell}}_{n,d} \to B^{\mathrm{ell}}$. From the perspective of this paper, the main difficulty in approaching Conjecture \ref{conj2} is to understand the small diagonal cycle for $M_{n,d}$ using parabolic moduli spaces.

Conjectures \ref{conj1} and \ref{conj2}, which concern multiplicative filtrations at the sheaf-theoretic and the motivic levels respectively, are further expected to be refined into multiplicative decompositions; see Remark \ref{rmk0}(b).

\subsection{Refined motivic decomposition}

Support theorems yield rich structures for the decomposition theorem associated with the Hitchin system. For example, by studying the decomposition theorem over the locus $B^{\mathrm{red}} \subset B$ of the Hitchin base formed by reduced spectral curves, de Cataldo, Heinloth, and Migliorini showed in \cite{dCHeM} that every partition of $n$ contributes a support of $Rf_*\BQ_{M_{n,d}}$. Therefore, following the philosophy of Corti--Hanamura \cite{CH}, we expect to have the motives which lift the contribution of the support indexed by each partition.

\begin{question}\label{Q1}
    Can we find algebraic cycles which yield a refinement of the motivic decomposition (\ref{motivic_DT}) based on supports of the decomposition theorem for $f: M_{n,d} \to B$?
\end{question}

The conjectural refinement of Question \ref{Q1} is mysterious for the following reasons. As we commented in Remark \ref{remark1.1}, if we consider the decomposition theorem for the $\mathfrak{D}$-twisted Hitchin system, the support theorem behaves very differently for $\mathfrak{D}=0$ and $\mathfrak{D}>0$. In particular, the decomposition theorem has full support when $\mathfrak{D}>0$. Therefore, the conjectural motives of Question \ref{Q1}, which lift the contribution of smaller supports of the decomposition theorem, should only exist for $\mathfrak{D} = 0$.

\subsection{Singular moduli spaces}
When the degree $d$ is not coprime to the rank $n$, the Hitchin moduli space $M_{n,d}$ is singular. Many properties concerning the decomposition theorem and the perverse filtration in the coprime setting are expected to be generalized to the non-coprime setting. We expect that there should exist algebraic cycles which induce these sheaf-theoretic structures. 

A challenging question in this direction is to construct relative \emph{BPS motives} for all $M_{n,d}$ without the coprime assumption. More precisely, from the perspective of cohomological Donaldson--Thomas theory the role played by the constant sheaf $\BQ_{n,d}$ should be replaced by a mysterious perverse sheaf, called the BPS sheaf
\[
\Phi_{\mathrm{BPS},d} \in \mathrm{Perv}(M_{n,d}),
\]
when $M_{n,d}$ is possibly singular. This perverse sheaf recovers the (shifted) constant sheaf $\BQ_{M_{n,d}}[\dim M_{n,d}]$ when $(n,d)=1$, but is in general more complicated than the intersection cohomology complex $\mathrm{IC}_{M_{n,d}}$; see \cite{Toda} or \cite[Section 0.4]{MS}. The sheaf-theoretic $\chi$-independence result (\ref{coh_chi}) is generalized by Kinjo--Koseki \cite{KK} using the BPS perverse sheaf:
\begin{equation}\label{chi_KK}
Rf_{d*} {\Phi_{\mathrm{BPS},d}} \simeq Rf_{d'*} \Phi_{\mathrm{BPS},d'}.
\end{equation}

\begin{question}
Can we lift the sheaf-theoretic $\chi$-independence identity (\ref{chi_KK}) to relative Chow motives over the Hitchin base $B$?
\end{question}
    
When $(n,d)\neq 1$, we have
\[
\Phi_{\mathrm{BPS},d} = \mathrm{IC}_{M_{n,d}} \oplus \varphi_d
\]
where $\varphi_d$ is a perverse sheaf supported on the singular locus of $M_{n,d}$. It would be interesting to see how the extra summand $\varphi_d$ is detected by algebraic cycles. The BPS motive which lifts $Rf_{d*} {\Phi_{\mathrm{BPS},d}}$ may be viewed as the motivic Gopakumar--Vafa invariant associated with the local Calabi--Yau 3-fold $X:= \mathrm{Tot}_\Sigma(\omega_\Sigma \oplus \CO_\Sigma)$.

\end{document}